\documentclass[11pt,leqno]{amsart}

\usepackage{latexsym}
\usepackage{amssymb}
\usepackage{amsmath}
\usepackage{amsthm}
\usepackage{verbatim}
\usepackage{pdfsync}
\usepackage[margin=1.2in]{geometry}

\usepackage{color}

\usepackage{tikz}
\usepackage[all]{xy}

\begin{document}




\newcommand{\A}{{\mathbb A}}
\newcommand{\C}{{\mathbb C}}
\newcommand{\F}{{\mathbb F}}
\newcommand{\G}{{\mathbb G}}
\newcommand{\R}{{\mathbb R}}
\newcommand{\Q}{{\mathbb Q}}
\newcommand{\X}{{\mathbb X}}
\newcommand{\Z}{{\mathbb Z}}
\newcommand{\HZ}{\widehat{\Z}}


\newcommand{\rom}[1]{\text{\rm #1}}
\renewcommand{\roman}{\rm}

\newcommand{\Aut}{\text{\rm Aut}}
\newcommand{\CH}{\widehat{\text{\rm CH}}}
\newcommand{\cha}{{\text{\rm char}}}
\newcommand{\CHe}{\text{\rm CHeeg}}
\newcommand{\degh}{\widehat{\text{\rm deg}}}
\newcommand{\degH}{\widehat{\text{\rm deg}}}    
\newcommand{\diag}{{\text{\rm diag}}}
\newcommand{\Diff}{\text{\rm Diff}}
\newcommand{\disc}{\text{\rm discr}}
\renewcommand{\div}{\text{\rm div}}
\newcommand{\divh}{\widehat{\text{\rm div}}}
\newcommand{\DS}{\text{\rm DS}}
\newcommand{\Ei}{\text{\rm Ei}}
\newcommand{\End}{\text{\rm End}}
\newcommand{\ev}{{\text{\rm ev}}}
\newcommand{\Gal}{\text{\rm Gal}}
\newcommand{\GL}{\text{\rm GL}}
\newcommand{\GSpin}{\text{\rm GSpin}}
\newcommand{\Hom}{\text{\rm Hom}}
\newcommand{\hor}{{\text{\rm horiz}}}
\newcommand{\id}{\text{\rm id}}
\newcommand{\im}{\text{\rm im}}
\renewcommand{\Im}{\text{\rm Im}}
\newcommand{\inv}{{\text{\rm inv}}}
\newcommand{\Jac}{\text{\rm Jac}}
\newcommand{\Leray}{{\mathrm L}}
\newcommand{\Lie}{\text{\rm Lie}}
\newcommand{\Mp}{\text{\rm Mp}}
\newcommand{\mult}{\text{\rm mult}}
\newcommand{\MW}{\text{\rm MW}}
\newcommand{\MWt}{\widetilde{\MW}}
\newcommand{\new}{\text{\rm new}}
\newcommand{\Nm}{\text{\rm Nm}}
\newcommand{\ord}{\text{\rm ord}}
\newcommand{\PGL}{\text{\rm PGL}}
\newcommand{\Pic}{\text{\rm Pic}}
\newcommand{\Pich}{\widehat{\text{\rm Pic}}}
\newcommand{\pr}{\text{\rm pr}}
\newcommand{\ra}{\text{\rm ra}}
\newcommand{\Rao}{\mathrm R}
\renewcommand{\Re}{\text{\rm Re}}
\newcommand{\sgn}{\text{\rm sgn}}
\newcommand{\sig}{\text{\rm sig}}
\newcommand{\SL}{\text{\rm SL}}
\newcommand{\SO}{\text{\rm SO}}
\newcommand{\Sp}{\text{\rm Sp}}
\newcommand{\Spec}{\text{\rm Spec}\, }
\newcommand{\Spf}{\text{\rm Spf}}
\newcommand{\supp}{\text{\rm supp}}
\newcommand{\Sym}{{\text{\rm Sym}}}
\newcommand{\tr}{\text{\rm tr}}
\newcommand{\type}{\text{\rm type}}
\newcommand{\Ver}{\text{\rm Vert}}
\newcommand{\vol}{\text{\rm vol}}
\newcommand{\Wald}{\text{\rm Wald}}


\newcommand{\Cal}{\mathcal}     

\newcommand{\AHH}{\hat{\Cal A}}   
\newcommand{\CHH}{\hat{\Cal C}}
\newcommand{\MM}{\Cal D}          
\newcommand{\MMb}{\MM^\bullet}
\newcommand{\ssplit}{\text{\bf split}}
\newcommand{\whcc}{\widehat{\Cal C}}
\newcommand{\CO}{\mathcal O}
\newcommand{\COH}{\widehat{\CO}}
\newcommand{\M}{\Cal M}
\newcommand{\OB}{\Cal O_B}
\newcommand{\XX}{\mathcal X}
\newcommand{\bXX}{\bar\XX}
\newcommand{\wc}{\hat{\Cal C}}
\newcommand{\wch}{\wc^{\text{\rm hor}}}
\newcommand{\ZZ}{\Cal Z}
\newcommand{\ZH}{\widehat{\Cal Z}}   
\newcommand{\Zh}{\widehat{\Cal Z}}
\newcommand{\ZZh}{\ZZ^{\text{\rm hor}}}
\newcommand{\ZZv}{\ZZ^{\text{\rm ver}}}
\newcommand{\ZZhh}{\Zh^{\text{\rm hor}}}
\newcommand{\ZZhv}{\Zh^{\text{\rm ver}}}


\newcommand{\nass}{\noalign{\smallskip}}
\newcommand{\snass}{\noalign{\vskip 2pt}}
\newcommand{\tent}[1]{ \vphantom{\vbox to #1pt{}} }   


\newcommand{\scr}{\scriptstyle}
\newcommand{\disp}{\displaystyle}

\font\cute=cmitt10 at 12pt
\font\smallcute=cmitt10 at 9pt
\newcommand{\kay}{{\text{\cute k}}}
\newcommand{\smallkay}{{\text{\smallcute k}}}

\renewcommand{\a}{\alpha}
\renewcommand{\b}{\beta}
\newcommand{\e}{\epsilon}
\renewcommand{\l}{\lambda}
\renewcommand{\L}{\Lambda}
\renewcommand{\o}{\omega}
\renewcommand{\O}{\Omega}
\renewcommand{\P}{\Phi}
\newcommand{\ph}{\varphi}
\newcommand{\phih}{\widehat{\phi}}
\newcommand{\wphi}{\widehat{\phi}}
\newcommand{\phit}{\widetilde{\phi}}
\newcommand{\s}{\sigma}
\newcommand{\vth}{\vartheta}


%

\newcommand{\Pt}{P}
\newcommand{\Ph}{\P}
\newcommand{\Pht}{\tilde \P}   
\newcommand{\Kt}{K}           
\newcommand{\Mt}{M}

\newcommand{\pht}{\widetilde{\phi}}
\newcommand{\It}{I}
\newcommand{\Jt}{\widetilde{J}}
\newcommand{\lt}{\widetilde{\l}}
\newcommand{\vp}{\varpi}

\newcommand{\bom}{{\boldsymbol{\o}}}
\newcommand{\hbom}{\widehat{\bom}}
\newcommand{\ff}{{\bold f}}
\newcommand{\fsp}{\boldsymbol{f}_{\!\rm sp}}
\newcommand{\fev}{\boldsymbol{f}_{\!\rm ev}}
\newcommand{\fb}{\boldsymbol{f}}
\newcommand{\J}{\und{J}'}
\newcommand{\JJ}{\bold J'}
\newcommand{\V}{\bold V}
\newcommand{\xx}{\bold x}

\newcommand{\g}{{\mathfrak g}}
\renewcommand{\H}{\mathfrak H}


\newcommand{\back}{\backslash}
\newcommand{\CT}[1]{\operatornamewithlimits{CT}_{#1}}
\renewcommand{\d}{\partial}
\newcommand{\db}{\bar\partial}
\newcommand{\dbar}{\bar{\partial}}
\newcommand{\gs}[2]{\langle \,#1,#2\,\rangle}
\newcommand{\Gt}{G}
\newcommand{\hfal}{h_{\text{\rm Fal}}}
\newcommand{\II}{\int^\bullet}
\newcommand{\isoarrow}{\ {\overset{\sim}{\longrightarrow}}\ }
\newcommand{\lisoarrow}{\ {\overset{\sim}{\longleftarrow}}\ }
\newcommand{\limdir}[1]{\underset{\underset{#1}{\rightarrow}}{\lim}}
\newcommand{\lan}{\operatorname{\langle}\hskip .5pt}
\newcommand{\ran}{\,\operatorname{\rangle}}
\newcommand{\lra}{\longrightarrow}
\newcommand{\doublelra}{\ {\overset{\scr\lra}{\scr\lra}}\ }
\newcommand{\nat}{\natural}
\newcommand{\notmid}{\mkern-5mu\not\mkern5mu\mid}
\newcommand{\Optoc}{\text{\rm Opt}(O_{c^2d},O_B)}
\newcommand{\psim}{\psi^{-}}
\newcommand{\qeq}{\ \overset{??}{=}\ }
\newcommand{\sh}{\sharp}
\newcommand{\thCH}{\theta^{\text{\rm ar}}}
\newcommand{\wht}{\widehat{\theta}}     
\newcommand{\triv}{1\!\!1}
\renewcommand{\tt}{\otimes}
\newcommand{\und}[1]{\underline{#1}}
\newcommand{\z}{z}  

\newcommand{\thMW}{\theta^{\text{\rm ar}}}
\newcommand{\tph}{\widetilde{\widehat\phi_1}}
\newcommand{\Pet}{\text{\rm Pet}}





\newcommand{\thing}{ \raisebox{-6.4pt}{$\tilde{\tilde{}}$}  }   
\newcommand{\smallthing}{ \raisebox{-4.4pt}{$\scr\tilde{\tilde{}}$}  }
\newcommand{\ttilde}[1]{\overset{\smash{\thing}}{#1}}
\newcommand{\smallttilde}[1]{\overset{\smash{\smallthing}}{#1}}
\newcommand{\downhookarrow}{\hbox{$\downarrow\hskip -6.1pt\raisebox{6pt}{$\cap$}$}}


\providecommand{\bysame}{\makebox[3em]{\hrulefill}\thinspace}   
\newcommand{\hfb}{\hfill\break}
\newcommand{\margincom}[1]{\marginpar{\bf\raggedright #1}}
\newcommand{\Sec}{\S}


\numberwithin{equation}{section}
\setcounter{section}{0}
\setcounter{MaxMatrixCols}{15}


\newtheorem{theo}{Theorem}[section]
\newtheorem{lem}[theo]{Lemma}
\newtheorem{prop}[theo]{Proposition}
\newtheorem{cor}[theo]{Corollary}
\newtheorem*{atheo}{Theorem A}
\newtheorem{conj}[theo]{Conjecture}
\newtheorem{rem}[theo]{Remark}      
\newtheorem{defn}[theo]{Definition}

\newcommand{\OK}{O_{\smallkay}}
\newcommand{\DI}{\mathcal D^{-1}}

\newcommand{\pre}{\text{\rm pre}}

\newcommand{\Bor}{\text{\rm Bor}}
\newcommand{\Rel}{\text{\rm Rel}}
\newcommand{\rel}{\text{\rm rel}}

\parindent=0pt
\parskip=10pt
\baselineskip=14pt

\newcommand{\cutter}{\medskip\medskip \hrule \medskip\medskip}

\newcommand{\bpm}{\begin{pmatrix}}
\newcommand{\epm}{\end{pmatrix}}
\newcommand{\be}{\begin{equation}}
\newcommand{\ee}{\end{equation}}
\newcommand{\dop}{\,{\bold\cdot}\,}

\newcommand{\crit}{\text{\rm crit}}

\newcommand{ \bgs }[2]{\gs{#1}{#2}}

\newcommand{\hZ}{\widehat{\Z}}
\newcommand{\Sing}{\text{\rm Sing}}
\newcommand{\T}{B}

\newcommand{\tL}{\tilde L}
\newcommand{\tV}{\tilde V}
\newcommand{\tx}{\tilde x}

\newcommand{\Th}{\text{\rm Th}}

\newcommand{\now}{\count0=\time 
\divide\count0 by 60
\count1=\count0
\multiply\count1 by 60
\count2= \time
\advance\count2 by -\count1
\the\count0:\the\count2}

\title{Another product for a Borcherds form}

\author{Stephen Kudla}

\centerline{\it\hfill\today:\ \now}

\maketitle

In a celebrated pair of papers \cite{borch95} and \cite{borch98}, Borcherds constructed meromorphic modular forms on the locally symmetric 
varieties associated to rational quadratic spaces $V$ of signature\footnote{Equivalently, Borcherds usually works with signature $(2,n)$.} $(n,2)$. 
More precisely, for an even lattice $M$ with respect to the symmetric bilinear form $(\ , \ )$,  there is a finite Weil representation $\rho_M$ 
of an extension $\Gamma'$ of $\SL_2(\Z)$ on 
the group algebra $S_M=\C[M^\vee/M]$, where $M^\vee$ is the dual lattice of $M$. A weakly holomorphic modular form $F$ 
of weight $1-\frac{n}2$ and type $\rho_M$ is an $S_M$-valued holomorphic function of $\tau\in \H$, the upper half-plane, with transformation 
law
$$F(\gamma'(\tau)) = (c\tau+d)^{1-\frac{n}2} \,\rho_M(\gamma')F(\tau),$$
for $\gamma'\in \Gamma'$, and with Fourier expansion of the form 
$$F(\tau) = \sum_{m} c(m)\, q^m, \qquad c(m) \in S_M,$$
where $m\in \Q$ and where there are only a finite number of nonvanishing  terms with $m<0$, i.e., $F$ is meromorphic at the cusp. 
Let $D$ be one component of the space of oriented negative $2$-planes in $V(\R)= M\tt_\Z \R$. 
Assuming that the $c(m)$ for $m\le 0$ lie in $\Z[M^\vee/M]$, Borcherds constructs a meromorphic modular form $\Psi(F)$ on $D$
of weight $\frac12\,c(0)(0)$ with respect to an arithmetic group $\Gamma_M$ in $\Aut(M)$. 
The divisor of $\Psi(F)$ is given explicitly in terms of the $c(m)$'s for $m<0$ and, most remarkably, in a suitable neighborhood of 
any point boundary component,  $\Psi(F)$ is given by an explicit infinite 
product. 

In the present paper, assuming that the rational quadratic 
space $V= M\tt_\Z\Q$ contains isotropic $2$-planes, 
we give another family of product formulas for $\Psi(F)$, each valid in a neighborhood of the $1$-dimensional 
boundary component associated to such a $2$-plane $U$.  

In the simplest case, suppose that $M = L$ is an even unimodular lattice of signature $(n,2)$ and that there is a Witt decomposition
\be\label{eq1} V = U + V_0 + U' \ee
of $V$ such that\footnote{In a common terminology, `$L$ splits two unimodular hyperbolic planes.'}
$$L = L_U + L_0 + L_{U'},$$
where  
$L_{U} = L\cap U$, $L_{U'}= L\cap U'$, and $L_0 = L\cap V_0$.   Note that $L_0$ is even unimodular and positive definite. 
In this case,  $S_L  = \C\,\ph_0$ is one dimensional, with basis vector $\ph_0$ associated to the zero element of 
$L^\vee/L$, 
and we can write the input form $F = F_o\,\ph_0$  where $F_o$ is scalar valued. Write $c(m) = c_o(m)\,\ph_0$. 
Also associated to the decomposition (\ref{eq1}) and a choice of basis $e_1$ and $e_2$ for $M_U$ and dual basis $e_1'$ and $e_2'$ for $U'$, 
is a realization of $D$ as a Siegel domain of the third kind:
$$D \simeq \{(\tau_1, \tau_2', w_0) \in \H \times \C \times V_0(\C)\mid 4 v_1 v_2' + Q(w_0-\bar w_0)>0\},$$
where $v_1= \Im(\tau_1)$, $v_2'= \Im(\tau_2')$ and $Q(x) = \frac12 (x,x)$. We write $q_1= e(\tau_1)$ and $q_2= e(\tau_2')$, where $e(t) = e^{2\pi i t}$.  
 In these coordinates, our product formula has the following form.   

\begin{atheo}
In a suitable neighborhood of the $1$-dimensional boundary 
component associated to $U$, the associated Borcherds form $\Psi(F)$ is the product of the factors 
\be\label{example1-first}
\prod_{\substack{a\in \Z\\ \snass a>0}} \prod_{b\in\Z}
\prod_{\substack{x_0\in L_0}} 
\big(1-q_2^a \,q_1^b\,e( - (x_0, w_0))\big)^{c_o(ab-Q(x_0))}
\ee
and 
\be\label{example1-second-II}
\kappa\,q_2^{I_0}\,\eta(\tau_1)^{c_o(0)}\,
 \prod_{\substack{x_0\in L_0\\ \snass x_0\ne0}} \bigg(\, \frac{\vartheta_1(-(x_0,w_0),\tau_1)}{\eta(\tau_1)}\,\bigg)^{c_o(-Q(x_0))/2}
\ee
where $\kappa$ is a scalar of absolute value $1$ and 
$$I_0 = -  \sum_{m} 
\sum_{x_0\in L_0} c_o(-m)\,\s_1(m-Q(x_0)).$$
\end{atheo}

Here $\eta(\tau)$ is the Dedekind eta-function,  $\vartheta_1(\tau,z)$ is the Jacobi theta function (\ref{jacobi-I}), and $\s_1$ is the usual divisor function extended 
by the conventions $\s_1(r)=0$ 
if $r\notin \Z_{\ge0}$ and $\s_1(0) = -\frac1{24}$.  
Note that in the product (\ref{example1-second-II}),  $x_0$ runs over a finite set. 
The result in the general case, i.e., for any $M$ and $U$, has a similar shape, cf.  Theorem~\ref{mainthm} and Corollary~\ref{gen-cor}  in section 2.
Note that the scalar $\kappa$ arises due to the fact that $\Psi(F)$ is only defined up to such a factor. 
Of course, if there are several inequivalent isotropic planes, it remains to determine how these factors vary. 

Our proof of the product formula is a variant of that of Borcherds \cite{borch98}. There he computes the 
regularized theta lift of $F$ in the tube domain coordinates associated to the maximal parabolic subgroup 
stabilizing an isotropic line. He observes that, in a suitable neighborhood of the cusp and up to terms ultimately 
arising from a Petersson norm, the regularized 
theta integral is the $\log|\cdot |^2$ of a holomorphic function on that neighborhood. Since, 
up to an explicit singularity along some special divisors, the regularized integral 
is globally defined and automorphic, Borcherds is able to conclude the existence of the meromorphic 
modular form $\Psi(F)$ with the given product expansions.  

Analogously, we compute the regularized theta lift in the (Siegel domain of the third kind) coordinates associated to the maximal parabolic 
subgroup stabilizing an isotropic $2$-plane $U$.  Again in a suitable neighborhood of the $1$-dimensional 
boundary component associated to $U$, we find that the regularized lift is the $\log|\cdot |^2$ 
of a meromorphic function with a product formula, as described in a special case in Theorem A. 
One main difference between our product and that of Borcherds is that our expression includes the finite product (\ref{example1-second-II}), defined on all of $D$,  of functions having zeros and poles 
in our neighborhood. In effect, this factor accounts for some of the singularities which limit the convergence of the classical Borcherds product
and require the introduction of Weyl chambers in the negative cone in its description. 
With these singularities absorbed in  (\ref{example1-second-II}), our product is valid in a much simpler region depending only 
on the Witt decomposition (\ref{eq1}) and the choice of a basis $e_1$, $e_2$ for $M_U$. 

The difference between the two products may be viewed as a reflection of the geometry.  Suppose that $\Gamma\subset \Aut(M)$ is a 
neat subgroup of finite index. Then in a smooth toroidal compactification $\widetilde{X}$ of $X = \Gamma\back D$,  
the inverse image of a $1$-dimensional boundary component in the Bailey-Borel compactification $X^{BB}$ 
is a Kuga-Sato variety over a modular curve. This component of the compactifying divisor arises from the fact that $\Gamma_U\back D$, where 
$\Gamma_U$ is the stabilizer of $U$ in $\Gamma$, 
can be viewed as a line bundle, minus its zero section, on such a Kuga-Sato variety. A compactifying chart is obtained by filling in the zero section. 
In our coordinates, the boundary component in $X^{BB}$ is the modular curve $\bar\Gamma_U\back \H$, where $\bar\Gamma_U$ is the 
subgroup of $\SL(U)$ obtained by restricting elements of $\Gamma_U$ to $U$ and $\tau_1\in \H$. The coordinate $w_0$ is the fiber coordinate
of the Kuga-Sato variety and $q_2 = e(\tau_2')$ is the fiber coordinate for the line bundle over it.   In particular, 
the product formula of Theorem A shows that $\Psi(F)$ extends to this compactifying chart provided $q_2^{I_0}$ does 
(this will depend on the intersection of $\Gamma$ with the center of $P_U$), and the order of vanishing of the extension 
along the compactifying divisor can be read off.  Since the factor (\ref{example1-first}) goes to $1$ as $q_2$ goes to zero, the (regularized) value 
of $\Psi(F)$ on the compactifying divisor is given by
\be\label{first-FJ}
\Psi_0(\tau_1,w_0) = \lim_{q_2\rightarrow 0} q_2^{-I_0}\,\Psi(F) = \kappa\,\eta(\tau_1)^{c_o(0)}\,
 \prod_{\substack{x_0\in L_0\\ \snass x_0\ne0}} \bigg(\, \frac{\vartheta_1(-(x_0,w_0),\tau_1)}{\eta(\tau_1)}\,\bigg)^{c_o(-Q(x_0))/2}.
 \ee
In contrast, the description of the inverse image in $\widetilde{X}$ of a point boundary component in $X^{BB}$ involves the machinery of torus embeddings, 
in particular the choice of a system of rational 
polyhedral cones in the negative light cone associated to an isotropic line, \cite{looijenga.IV}. The classical Borcherds products, 
which depend on the choice of a Weyl chamber, should give a description of $\Psi(F)$ in the various associated coordinate charts. 
The combinatorics in this situation are considerably more complicated than those required for the $1$-dimensional boundary components.
It is also worth noting that Bruinier and Freitag \cite{bruinier.freitag}   investigated the behavior of Borcherds products locally in a neighborhood of a generic point of 
a rational  $1$-dimensional boundary component and that the factor (\ref{example1-second-II}) in Theorem~A is closely related to what they call a local 
Borcherds product, cf.~section~\ref{LBP}  below. 

Product formulas like that of Theorem A already occur in Borcherds \cite{borch95} and in work of 
Gritsenko \cite{gritsenko.Jacobi-n}. Indeed, in Borcherds original approach and in the construction of \cite{gritsenko.Jacobi-n},  the input data is a suitable Jacobi form and  
the associated modular form for an arithmetic subgroup $\Gamma$ in $O(n,2)$ is constructed by applying an infinite sum of Hecke 
operators to it, cf.~the discussion on pp.191--2 of \cite{borch95}, 
especially the third displayed equation on p.192. This method requires information about the generators for $\Gamma$ and the theory of Jacobi forms.
The method of regularized theta integrals developed by Borcherds in his subsequent paper \cite{borch98},  stimulated by ideas of 
Harvey and Moore \cite{harvey.moore},  takes a vector valued form $F$ as discussed above as input and works greater generality. 
In particular,  the modularity 
of the output ultimately follows from the transformation properties of the theta kernel involved. 

Our product formula can be viewed as providing an analogue of the expressions arising in \cite{borch95} and \cite{gritsenko.Jacobi-n} in the general case.
In the case of a unimodular lattice $L$ as in Theorem~A,  we have
\be\label{grit.prod}
\Psi(F)(w) = q_2^{I_0}\,\Psi_0(\tau_1,w_0)\, \exp\bigg( -\sum_{n=1}^\infty \frac{1}{n}  
\sum_{a=1}^{\infty} q_2^{an}\,\Theta_{a,n}(F)(\tau_1,w_0)\ \bigg).
\ee
where
\be
\Theta_{a,n}(F)(\tau_1,w_0) = \sum_{m} c_o(m)\,q_1^{a^{-1}mn}\ \sum_{\substack{x_0\in L_0\\ \snass a\mid (Q(x_0)+m)}} 
q_1^{a^{-1}n Q(x_0)}\,e(-(x_0,w_0)).\ee
Note that one obtains the Fourier-Jacobi expansion of $\Psi(F)$ by expanding the exponential function; for example, the next such coefficient is 
$\Psi_0\cdot \Theta_{1,1}(F)$. 
The analogue of (\ref{grit.prod}) for any Borcherds lift $\Psi(F)$ is given in Corollary~\ref{GNformula} which thus shows that 
every Borcherds lift has such a product.

As already explained, our construction is based on the method of regularized theta integrals and makes no use of  the theory of 
Jacobi forms or of generators for $\Gamma$. It is amusing to note that the eta-function and Jacobi theta function 
come into our formula due to the first and second Kroecker limit formulas which turn up in our calculation precisely in the form
discussed in \cite{siegel}.  The infinite sum of Hecke operators occurring in \cite{borch95} and \cite{gritsenko.Jacobi-n} is implicit 
in our computation as well, for example in the non-singular orbits in (\ref{etareps}), but we have not tried to include this in our formulation.

We now discuss the contents of various section.  Section 1 sets up the notation, in particular the realization of $D$ as a Siegel 
domain of the third kind determined by a Witt decomposition (\ref{eq1}) for an isotropic $2$-plane $U$.  We also explain a convenient 
choice of a sublattice $L\subset M$ compatible with (\ref{eq1}).  The main calculations are then done for $S_L$-valued forms $F$.
In section 2, we review the regularized theta integral construction of the Borcherds form $\Psi(F)$ and state the first form of our product formula (Theorem~\ref{mainthm}). 
Then we give a more intrinsic description of the index sets which yields a formula for general lattices 
$M$. The final formula depends only on $M$, the choice of Witt decomposition (\ref{eq1}), and the choice of a basis $e_1$ and $e_2$ 
for $M\cap U$.  In section 3, which is the technical core of the paper, 
we compute the regularized theta integral.  The key point is to express the theta kernel in terms of a mixed model 
for the Weil representation determined by the Witt decomposition (\ref{eq1}).  From a classical point of view, this amounts to taking 
a certain partial Fourier transform of theta kernel.  Precisely the same trick is an essential part of Borcherds' calculation in section~7 of \cite{borch98}, 
where the relevant Witt decomposition involves an isotropic line.  In the mixed model, the theta integral has an orbit decomposition 
(\ref{theta.decompo-I}) which allows a further unfolding argument.  There are non-singular terms, terms of rank 1, and the zero orbit, 
and these eventually give rise to the various factors in Theorem~\ref{mainthm}.  The calculation for the rank $1$ orbits is very pleasant, 
as it leads almost immediately to precisely the expressions evaluated by means of the first and second Kronecker limit formulas in 
Siegel \cite{siegel}. The contribution of the zero orbit is already essentially determined by Borcherds. 
It is worth noting that in most of our calculation, we use the coordinates on $D$ that come from the action of the real points of the 
unipotent radical of the maximal parabolic $P_U$, whereas the natural complex coordinates involve a shift (\ref{deftau2prime}). 
To get our final product formula expressed in these holomorphic coordinates, we need to combine the contribution of the zero orbit with some of the 
factors occurring in the Kronecker limit formula terms, cf.~(\ref{vanishing}) and (\ref{extra-junk}). That this is possible depends 
essentially on the identity of Proposition~\ref{borch-quad-rel-II} (Borcherds' quadratic identity), which seems to lie at the 
heart of the theory of Borcherds forms, cf.~the comments on p.536 of \cite{borch98} and 
Lemma~2.2 of \cite{GK-II}, for example.  In section 4, we check that our formula yields several examples from the literature. 
For more recent work using the Jacobi form method 
cf. Cl\'ery-Gritsenko \cite{clery.gritsenko} and the references given there. In section 5, we explain how to pass from our product formula to 
one of those given by Borcherds {\it for a particular choice of Weyl chamber}.  In this case, the Weyl vector in the Borcherds product arises 
in a natural way from the factors in our formula. The Borcherds products for other Weyl chambers do not seem to be accessible in this way.

This paper is the outcome of a question raised in discussions with Jan Bruinier, Ben Howard, Michael Rapoport, and Tonghai Yang in Bonn 
in June of 2013.  I would like to thank them for their interest and encouragement.

\section{Complex coordinates and lattices}

\subsection{The Siegel domain of the third kind}\label{siegel-3rd}
Let $V$ be a rational quadratic space of signature $(n,2)$ and fix a Witt decomposition (\ref{eq1}),
with $\dim U=2$.  Choose a basis $e_1$, $e_2$ for $U$ and dual basis $e_1'$, $e_2'$ for $U'$, 
and write $x = x_0+ x_{11}e_1' + x_{12} e_2' + x_{21} e_1 + x_{22} e_2$ 
as 
\be\label{x-coord} 
x = \bpm x_2\\ x_0\\ x_1\epm \in V,
\ee
with $x_1$, $x_2\in \Q^2$ (column vectors) and $x_0\in V_0$. Then 
$$(x,x) = (x_0,x_0) + 2\, x_1\dop x_2,$$
where the second term is the dot product. The unipotent radical of the parabolic subgroup $P_U$ of $G= O(V)$ stabilizing $U$ is 
$$n(b,c) = \bpm 1_2& -b^* & cJ- Q(b)\\ {}& 1_{V_0}& b\\ {}&{}&1_2\epm,$$ 
where $b = [b_1,b_2] \in V_0^2$, $c\in \R$, and 
$$J=\bpm {}&1\\ -1&{}\epm.$$    
Here $b^*$ is the element of $\Hom(V_0,\Q^2)$ defined by 
$$ b^*(v_0) = \bpm (b_1,v_0)\\ (b_2,v_0)\epm,$$
and $Q(b) = \frac12 ((b_i,b_j))$.  In particular, 
\be\label{eq2}
n(b,c) x = \bpm x_2 - (b,x_0) + (cJ-Q(b)) x_1\\ x_0 + b x_1 \\ x_1 \epm. \ee
The Levi factor of $P_U$ determined by (\ref{eq1}) is
$$M_U \simeq GL(U) \times O(V_0).$$
Here, for example, if $\a\in \GL_2$ and $h\in O(V_0)$, we have 
$$m_U(\a,h) x = \bpm \a x_2 \\ h\, x_0 \\ {}^t\a^{-1} x_1\epm.$$

We review the realization of  the space of oriented negative $2$-planes in $V(\R)$,  as a Siegel domain 
of the third kind associated to the Witt decomposition (\ref{eq1}).  For a more elegant treatment, cf. \cite{looijenga.IV}. 
First recall that for an oriented negative $2$-plane $z$, we can view the orientation as a complex structure $j_z$ on $z$ preserving the inner product. 
The isomorphism of the space of oriented negative $2$-planes with 
\be\label{Q-model}
 \{\ w\in V(\C)\mid (w,w)=0, (w,\bar w)<0\ \}/\C^\times \quad \subset\  \mathbb P(V(\C))
\ee
is realized by sending $z$ to $w(z)$, the $+i$-eigenspace of $j_z$ on the complexification $z_\C$. 

Note that $U^\perp = V_0+U$ is positive semidefinite, so the projection of $V$ to $U'$ with kernel $V_0+U$ 
induces an isomorphism of any negative $2$-plane $z$ with $U'$.  In particular, an orientation of $z$ induces an orientation on $U'$
 and on $U$. 
The projections $\pr_{U'}(w(z))$ and $\pr_{U'}(\bar w(z)) \in U'_\C$ 
form a basis, so that, up to scaling, we can write
$$ w = \bpm u \\ w_0 \\ \tau_1\\ 1\epm,\qquad u\in \C^2, \ w_0\in V_{0}(\C), \ \tau_1\in \C-\R.$$
We assume that the orientations are chosen so that $D$ is the component for which $\tau_1\in \H$ and we write $Q$ for the corresponding 
component of (\ref{Q-model}). 

For a pair $\tau_1$ and $\tau_2\in \H$, let 
$$w(\tau_1,\tau_2) = \bpm -\tau_2\\ \tau_1\tau_2\\ 0\\ \tau_1\\ 1\epm  = \bpm -\tau_2 J \\ 0 \\ 1_2\epm \bpm \tau_1\\1\epm.$$
Note that 
$(w,\bar w) = -4 v_1 v_2.$
In particular, $|y|^2 = 2v_1v_2$ in the notation of \cite{Bints}, (1.10), p.299.
Then there is an isomorphism
\be\label{gpactioncoord}
i:   \H \times \H \times V_0^2(\R)\isoarrow Q,
\ee
defined by 
$$ i(\tau_1,\tau_2,v_0) = n(v_0,0)\cdot w(\tau_1,\tau_2) = \bpm \big(\ -\tau_2\,J - Q(v_0)\ \big) \bpm \tau_1\\1\epm \\ v_0\bpm \tau_1\\ 1\epm \\ \tau_1\\ 1\epm
 = \bpm  -\tau_2 - \frac12(v_{01}, w_0)\\ \nass\nass \tau_1\tau_2 - \frac12(v_{02}, w_0)\\ \nass w_0 \\  \tau_1\\ 1 \epm, $$
 where $w_0= v_0\bpm \tau_1 \\ 1 \epm$.   Note that the top entries do not depend holomorphically on $w_0$.  The problem is that $\tau_2$ is not 
 a natural holomorphic coordinate when $w_0\ne0$. To fix this, we write our vector as
 \be\label{w-coords}
  w= \bpm -\tau_2'\\ \tau_1\tau_2' - Q(w_0) \\ w_0 \\ \tau_1\\1\epm,
  \ee
 where
 \be\label{deftau2prime}
 \tau_2^{\prime}= \tau_2^{\phantom{\prime}}+\frac12 (v_{01}, w_0)\in \C.
 \ee
This then satisfies $(w,w)=0$ and varies holomorphically with $\tau_1\in \H$, $\tau_2'\in \C$ and $w_0\in V_{0}(\C)$, subject to 
\be\label{siegel3}
0> (w,\bar w) =  - Q(w_0-\bar w_0) - 4 v_1 v_2'  . 
\ee
Since $Q(w_0-\bar w_0) = - 4 v_1^2 Q(v_{01})$, this just amounts to the condition 
\be\label{siegel-II}
v_2'> v_1 Q(v_{01}).\ee

\begin{rem} In the case of signature $(3,2)$ and quadratic form 
\be\label{sig-32}
\bpm {}&{}&1_2\\{}&2&{}\\ 1_2&{}&{}\epm,
\ee
we have $Q(v_{01})=v_{01}^2$ and condition (\ref{siegel-II}) just says that 
$$\bpm \tau_1&w_0\\ w_0&\tau_2'\epm \in \H_2,$$
\end{rem}
the Siegel space of genus $2$. 

In the calculations that follow, we have chosen to work with the `group action' coordinates 
(\ref{gpactioncoord}) and to recover the `holomorphic coordinates' 
(\ref{w-coords})
by a substitution at the end. One could, alternatively, work with the holomorphic coordinates throughout.  

\subsection{Boundary components}  Let $\bar Q$ be the closure of $Q$ 
in $\mathbb P(V(\C))$ (in the complex topology), and note that the set $\d Q = \bar Q - Q$ consists of certain isotropic lines in $V(\C)$. 
The rational point boundary components in $\d Q$ are the isotropic lines in $V(\Q)$. If $U\subset V$ is an isotropic plane, then the associated 
rational $1$-dimensional boundary component is the set 
$$\Cal C(U)= \{ \ w\in U(\C)\mid  U(\C) = \text{span}\{w, \bar w\} \ \}/\C^\times \quad \subset \d Q.$$
If a basis $e_1$, $e_2$ for $U$ is chosen, then there is an isomorphism
$$\mathbb P^1(\C) - \mathbb P^1(\R) \isoarrow \Cal C(U), \qquad \tau_1\mapsto \C (\tau_1 e_2 - e_1),$$
and the rational isotropic lines in $U$ correspond to points of $\mathbb P^1(\Q)$ and to the rational point boundary 
components in the closure of $\Cal C(U)$. For a choice of $U'$ with dual basis $e_1'$ and $e_2'$, we have Siegel domain coordinates 
$(\tau_1,\tau_2',w_0)$ as above, and, as $v_2'$ goes to infinity, the line in $Q$ spanned by the vector $w$ given by (\ref{w-coords}) goes to 
the isotropic line $\C(\tau_1 e_2- e_1)$ in $\Cal C(U)$. 
Finally, for a point $\C(\tau e_2 - e_1	)$ in $\Cal C(U)$, a basis for the open neighborhoods in the Satake 
topology\footnote{Stricly speaking, we are describing the intersection of such an open set with $Q$.} is given by 
$$\{(\tau_1,\tau_2',w_0)\in Q\mid |\tau_1-\tau|<\e_1, w_0\in V_0(\C),  Q(w_0-\bar w_0)  + 4 v_1 v_2'  >\frac{1}{\e}\},$$
for $\e_1>0$ and $\e>0$, cf., for example, \cite{bruinier.freitag}, p.10, or  \cite{looijenga.IV}, p.542.

\subsection{Lattices}\label{subsection-lattice}  Suppose that $M$, $(\ ,\ )$,  is an even integral lattice in $V$ with dual lattice 
$M^\vee\supset M$.  Let $S_M\subset S(V(\A_f))$ be the subspace of functions supported in $M^\vee\tt_\Z\hat\Z$ which 
are translation invariant under $\hat M = M\tt_\Z\hat \Z$. This space is spanned by the characteristic functions $\ph_\l$ of the cosets 
$\l+\hat M$, for $\l\in M^\vee/M$.  Note that, if $L\subset M$ is a sublattice, then $S_M\subset S_L$. 

For a given Witt decomposition (\ref{eq1}), we construct a compatible sublattice $L$ of $M$ as follows. 
Note that 
$$M\ \supset \ M_{U'} + M_0 + M_U,$$
where $M_U = M\cap U$, $M_{U'} = M\cap U'$ and $M_0 = M\cap V_0$. 
Let 
$$M_U^\vee = \{ u\in U\mid (u,M_{U'})\in \Z\},$$
so that $M_U^\vee \supset M_{U}$ and define $M_{U'}^\vee \supset M_{U'}$ analogously. Let $N\in \Z_{>0}$ 
be\footnote{We could require that $N$ be the smallest such integer.} 
such that $N \cdot M_{U'}^\vee \subset M_{U'}$, and let 
\be \label{def-L}
L = N \cdot M_{U'}^\vee + M_0 + M_U  = L_{U'} + L_0 + L_U.
\ee
Then,
\be\label{dual-L}
L^\vee = N^{-1} L_{U'} + L_0^\vee + N^{-1}L_U,
\ee
and, taking $e_1$ and $e_2$  a basis for $L_U$, with dual basis $e_1'$ and $e_2'$ for $U'$, as before,  in our coordinates (\ref{x-coord}), 
$x$ will be in $L$ for $x_2\in \Z^2$, $x_0\in L_0$ and $x_1\in N\Z^2$.

Let $\Gamma_M$ be the subgroup of $\Aut(M)$ that acts trivially on $M^\vee/M$ and define $\Gamma_L$ analogously. 
Since 
$$L\subset M\subset M^\vee \subset L^\vee,$$
we have $\Gamma_L\subset \Gamma_M$ of finite index. 
Thus, automorphic forms on $D$ with respect to $\Gamma_M$ can be viewed as automorphic forms with respect to $\Gamma_L$ 
with some additional conditions.  We will sometimes work with a neat subgroup $\Gamma\subset \Gamma_M$ of finite index. 
This allows us to avoid orbifold issues when discussing the geometry.

\section{Theta series and the Borcherds lift}

\subsection{The Borcherds lift} In working with the Borcherds lift, we use the adelic setup and notation of \cite{Bints} to which 
we refer the reader for unexplained notation. In particular, $G'_{\A}$ (resp. $G'_\R$) 
is the metaplectic cover of $\SL_2(\A)$ (resp, $\SL_2(\R)$) and $\Gamma'$ is the inverse image of $\SL_2(\Z)$ in $G'_\R$ . 

The input to our Borcherd lift will be a weakly holomorphic modular form $F$ on $G'_\A$ valued in $S_M$ of weight $\ell = 1-\frac{n}2$
whose Fourier expansion is 
\be\label{F-tau}
F(g'_\tau) = v^{-\ell/2}\sum_{m} c(m)\,q^m,
\ee
where $c(m) \in S_M$. For any sublattice $L\subset M$, we  can write 
\be\label{cm-l} 
c(m) = \sum_{\l\in L^\vee/L} c_\l(m) \,\ph_\l
\ee
with respect to the coset basis $\ph_\l$ for $S_L$.  
For an oriented negative $2$-plane $z\in D$, we let 
$$(x,x)_z = (x,x) + 2 R(x,z), \qquad R(x,z) =|(\pr_z(x), \pr_z(x))|, $$
be the corresponding majorant, where $\pr_z(x)$ is the $z$-component of $x$ with respect to the 
decomposition $V(\R) = z^{\perp} + z$. Let
$$\ph_{\infty}(x,z) = \exp(-\pi (x,x)_z),$$ 
be the corresponding Gaussian.  For a Schwartz function $\ph\in S(V(\A_f))$  and $\tau\in \H$, there is a theta function 
\be\label{theta.fun}
\theta(g'_\tau,\ph_\infty(z)\ph) = \sum_{x\in V(\Q)} \o(g'_\tau)\ph_\infty(x,z)\,\ph(x).
\ee
We can view this as defining a family of distributions $\theta(g'_\tau,\ph_\infty(z))$ on $S(V(\A_f))$, depending on $\tau$ and $z$, and it will be convenient 
to write $\gs{\ph}{\theta(g'_\tau,\ph_\infty(z)}$ for the pairing of such a  distribution with $\ph$. 
Pairing with the $S(V(\A_f))$-valued function $F$,  
we get an $\SL_2(\Z)$-invariant function
$\bgs{F(g'_\tau)}{\theta(g'_\tau,\ph_\infty)}$
on $\H$. 
We want to compute the regularized theta lift
\begin{equation}\label{reg-int}
\P(z;F)=\int_{\Gamma'\back \frak H}^{\text{reg}} \gs{F(g'_\tau)}{\theta(g'_\tau,\ph_\infty(z))}\,
v^{-2}\,du\,dv
\end{equation}
in the coordinates of section~\ref{siegel-3rd} associated to a $1$-dimensional boundary component. 

Recall that the regularization used by Borcherds is defined as follows. 
Let $\xi$ be a $\Gamma'$ invariant (smooth) 
function on $\frak H$, satisfying the following two conditions:
\begin{enumerate}
\item
There exists a constant $\s$ such that the limit
$$\phi(s,\xi)=\lim_{T\rightarrow\infty} \int_{\Cal F_T} \xi(\tau)\, v^{-s-2}\,du\,dv$$
exists for $\text{\rm Re}(s)>\s$ and defines a holomorphic function 
of $s$ in that half plane.
\item
The function $\phi(s,\xi)$ has a meromorphic continuation
to a half plane $\text{\rm Re}(s)>-\e$ for some $\e>0$.
\end{enumerate}
Then the regularized integral
$$\int_{\Gamma'\back\frak H}^{\text{reg}}  \xi(\tau)\, v^{-2}\, du\,dv$$
is defined to be the {constant term} of the Laurent expansion of $\phi(s,\xi)$ at $s=0$. 

\subsection{Another product formula}  

One of Borcherds' main results in \cite{borch98} is that  the regularized theta integral $\P(z;F)$ 
can be written as 
\be\label{borch.form}
\P(z;F) = -2\log|\Psi(z;F)|^2 - c_0(0)\big(\ \log|y|^2 + \log(2\pi) -\gamma),
\ee
where 
$\Psi(F)$ is a meromorphic modular form of weight  $c_0(0)/2$ on $D$ and $y$ is the imaginary part of 
$z$ in a tube domain model associated to an isotropic line.  In a suitable neighborhood of the corresponding point rational boundary 
component, he shows that $\Psi(z;F)$ has a product expansion. 
Our main result is another product expansion for $\Psi(z;F)$, valid in a neighborhood of a $1$-dimensional 
rational boundary component.  We will explain the relation between the two products in section~\ref{section-compare}. 
Our computation is, in fact, quite analogous to that given in \cite{borch98} and, as a byproduct, we also 
derive (\ref{borch.form}) and another proof of the existence of $\Psi(z;F)$. 

Here is our main result. 

\begin{theo} \label{mainthm}
Suppose that the lattice $L$ is chosen as in section~\ref{subsection-lattice} and that 
$$F(\tau) = \sum_{m} c(m) \, q^m$$ 
is a weakly holomorphic $S_L$-valued modular form of weight $-\ell = 1-\frac{n}2$,  type $\rho_L$ and integral coefficients for $m\le 0$.  
There are positive constants $A$ and $B$, depending on $F$ and on the Witt decomposition (\ref{def-L}), cf. Lemma~\ref{lem-tedious}, such that 
in a region of the form  
$$v_2' > (A+Q(v_{01}))v_1 + B v_1^{-1},$$
the Borcherds form $\Psi(F)$  is  the product of the three factors:\hfb
(i)
$$
\prod_{\substack{\l\\ \snass \l_{12}=0}}\prod_m \bigg( \prod_{\substack{a\in \l_{11}+N\Z\\ \snass a>0}}
\prod_{\substack{x_0\in \l_0+L_0\\ \snass
a\mid (m+Q(x_0)+a\l_{21})}} 
\big(1-q_1^b\, q_2^a\,e( - (x_0, w_0) -\L_2)\big) \bigg)^{c_\l(m)},$$
where $q_1= e(\tau_1)$, $q_2=e(\tau_2')$, $b = a^{-1}(m+Q(x_0)+ a\,\l_{21})$ and $\L_2 = \l_{21}\tau_1+\l_{22}$, \hfb
(ii)
$$\prod_m\prod_{\substack{\l \\ \snass \l_1=0}}
\bigg(\ \prod_{\substack{x_0\in \l_0+L_0\\ \snass Q(x_0)=m}}  \frac{\vartheta_1(-(x_0,w_0)-\L_2,\tau_1)}{\eta(\tau_1)}\,e(\,(x_0,w_0) + \frac12\,\L_2\,)^{\l_{21}}
\ \bigg)^{c_\l(-m)/2},
$$
where the factor for  $m=0$ and $\l=0$ is omitted,  
and\hfb
(iii) 
$$\kappa\, \eta(\tau_1)^{c_0(0)}\,q_2^{I_0},$$
where $\kappa$ is a constant of absolute value $1$ and 
$$I_0 = -  \sum_{m} \sum_{\substack{\l \\ \snass \l_1=0}}
\sum_{x_0\in \l_0+L_0} c_\l(-m)\,\s_1(m-Q(x_0)),$$
with $\s_1(0) = -\frac1{24}$ and $\s_1(r)=0$ for $r\notin \Z_{\ge0}$. 
\end{theo} 

\begin{rem}
 (1) The factor (i) converges and, in particular, has no zeroes or poles in our region near the boundary.  Moreover, its limit 
as $q_2\lra 0$ is $1$.\hfb 
(2) The finite product in factor (ii)
is independent of $\tau_2'$, and, in expanded form (\ref{factor-II-expanded}),  has evident zeroes or poles on the set of $(w_0, \tau_1)$'s where $(x_0,w_0)+\L_2=0$ for some $x_0\in L_0$ 
with  $c_{\l}(-Q(x_0))\ne 0$. 
The regularized integral itself is actually {\it finite} on these `walls'. This is because, just as in Borcherds' case, 
the integral is `over-regularized'.  Its values on the walls can be computed by using 
the expression in (\ref{rank1-sing}) to calculate the contribution of each term (\ref{basic-II-0}) for which $(x_0,w_0)+\L_2=0$. 
We omit this calculation.\hfb
(3) Finally, the factor (iii) gives the order of the pole or zero of the Borcherds form along the compactifying divisor, whose 
(semi-)local equation is $q_2=0$.  The regularized value along this divisor, obtained by removing the factor $q_2^{I_0}$,  is given by 
the product of theta functions in factor (ii) and the factor $\kappa\,\eta(\tau_1)^{c_0(0)}$. 
\end{rem}

\subsection{A more intrinsic variant}

In the statement of Theorem~\ref{mainthm}, we have written our product formula more or less in the expanded form that 
arises from the computations of section 3.  We next describe an alternative, more intrinsic version. 

First note that if $x\in \l+L$ with $(x, e_2)=0$, then $\l_{12}=0$, and we have
\be \label{x-coords} 
x = x_0 + a e_1' + (\l_{21}-b)\,e_1 + (\l_{22} - c)\,e_2 ,
\ee
where $x_0\in \l_0+L_0$, $a\in \l_{11}+ N\Z$, $b$, $c \in \Z$. 
Then, for $w$ as in (\ref{w-coords}),
$$-(x,w) = a\tau_2'+b\tau_1 + c - (x_0,w_0) - \l_{21}\tau_1-\l_{22},$$
and $e(-(x,w))$ is independent of $c$. Note that $\Z e_2 = L\cap \Q e_2$.

Therefore the factor in (i) of Theorem~\ref{mainthm} can be written as
\be\label{intrinsic-(i)}
\prod_{\substack{x\in L^\vee\\ \snass (x,e_2)=0\\ \snass (x,e_1)>0\\ \snass \mod L\cap \,\Q\, e_2}} \big(1-e(-(x,w))\ \big)^{c(-Q(x))(x)}.
\ee
Here, recall that $c(m)\in S_L\subset S(V(\A_f))$ so that $c(-Q(x))(x)$ is simply the value of the Schwartz function $c(-Q(x))$ at $x$, i.e., is
$c_\l(-Q(x))$ if $x\in \l+L$  and $0$ otherwise.  
The expression (\ref{intrinsic-(i)}) depends only on the choice of $U$ and of the basis $e_1$, $e_2$ for $L\cap U$. 
This choice of basis might be viewed as the analogue in our situation of the choice of Weyl chamber which occurs in the 
standard Borcherds product. 

The factor in (ii) of Theorem~\ref{mainthm} also has a more intrinsic expression. 
First we examine the range of the product. 
Recall that the isotropic $2$-plane $U$ determines a filtration $0\subset U\subset U^\perp \subset V$. 
A vector $x\in L^\vee$ lies in $L^\vee\cap U^\perp$ precisely when it is given as in (\ref{x-coords}) with $a=0$.
The vector $x$ then lies in $L^\vee \cap U$ precisely when $Q(x)=Q(x_0)=0$, since this condition implies that $x_0=0$. 

For a given $x \in L^\vee \cap U^\perp$, we have
$\l_{21} = (x, e_1')$, $\l_{22} = (x, e_2')$, and 
$$\L_2 = 
(w,e_2)^{-1}\big(\ (x,e_1')(w,e_1) + (x,e_2')(w,e_2)\ \big) 
= (w, e_2)^{-1} (x_U,w),$$
where $(x_U, w)$ is the pairing of the $U$-component $x_U$ of $x$ with the $w$, a quantity which, 
for a given $x$ and $w$,  
depends only on the Witt decomposition and not on the choice of basis $e_1$, $e_2$. 
Here we have written an expression for $\L_2$ that does not depend on the normalization $(w,e_2)=1$
of $w$.

Retaining the normalization $(w,e_2)=1$, the factor in (ii) can be written as the product of two factors, 
\be\label{intrinsic-(ii)-1}\prod_{\substack{x\in L^\vee \cap U^\perp\\  \snass \mod L\cap U\\ \snass Q(x)\ne0}} 
\bigg(\ \frac{\vartheta_1(-(x,w),\tau_1)}{\eta(\tau_1)} \, e((x,w)-\frac12 (x_U,w))^{(x,e_1')}\ \bigg)^{c(-Q(x))(x)/2},
\ee
and a factor arising from $x$ with $Q(x)=0$, i.e., $x_0=0$, so that $x=x_U$, 
\be\label{intrinsic-(ii)-2}
\prod_{\substack{x\in L^\vee \cap U\\  \snass \mod L\cap U\\ \snass x\ne 0}} 
\bigg(\ \frac{\vartheta_1(-(x,w),\tau_1)}{\eta(\tau_1)} \, e(\frac12(x,w))^{(x,e_1')}\ \bigg)^{c(0)(x)/2}
\ee
We have separated out the factor (\ref{intrinsic-(ii)-2}) since it depends only on $\tau_1$. 

In both (\ref{intrinsic-(ii)-1}) and (\ref{intrinsic-(ii)-2}) a square root has been taken, since it is only assumed that the Fourier coefficients $c_\l(-m)$ of $F$ for $m\in \Z_{>0}$ 
are integers. On the other hand, we know that $c_\l(-m) = c_{-\l}(-m)$ for all $m$. Thus we can choose a particular square root in (\ref{intrinsic-(ii)-1})
as follows. 
Let 
\be\label{V0-roots-1} 
R_0(F)= \{ \a_0\in L_0^\vee\mid Q(\a_0)>0, \ c(-Q(\a_0))(\a_0+\a_2)\ne 0\ \text{for some $\a_2\in L_U^\vee$} \}. 
\ee
These are precisely the $x_0$ components of vectors $x$ that appear in the product (\ref{intrinsic-(ii)-1}). 
Let $W_0$ be a connected component of the complement of the hyperplanes, $\a_0^\perp$, $\a_0\in R_0(F)$, in $V_0(\R)$. 
We refer to $W_0$ as a Weyl chamber in $V_0(\R)$. 

Then we can write (\ref{intrinsic-(ii)-1}) as 
\be\label{intrinsic-(ii)-1b}
\pm i^*\prod_{\substack{x\in L^\vee \cap U^\perp\\  \snass \mod U\cap L\\ \snass (x,W_0)>0}} 
\bigg(\ \frac{\vartheta_1(-(x,w),\tau_1)}{\eta(\tau_1)} \, e((x,w)-\frac12 [x,w]_U)^{(x,e_1')}\ \bigg)^{c(-Q(x))(x)},
\ee
where the sign depends on the choice of square roots in (\ref{intrinsic-(ii)-1}) and 
$$* = \sum_{\substack{x\in L^\vee \cap U^\perp\\  \snass \mod U\cap L\\ \snass (x,W_0)>0}} c(-Q(x))(x).$$
A change in the choice of $W_0$ simply changes (\ref{intrinsic-(ii)-1b})  by a sign. 

Next recall that, for any even integral lattice $M\subset M^\vee$ and Witt decomposition (\ref{eq1}), we have associated, in section~\ref{subsection-lattice}, 
a lattice $L\subset M$ that is compatible with the Witt decomposition, so that (\ref{def-L}) and (\ref{dual-L}) hold.   Note that, by construction, 
$L\cap U= M\cap U$ and hence $L\cap \Q e_2 = M\cap \Q e_2$.  Since $S_M\subset S_L$, a 
weakly holomorphic form $F$ valued in $S_M$ can be viewed as a weakly holomorphic form valued in $S_L$ and, in a neighborhood of the $1$-dimensional boundary 
component associated to $U$,  the Borcherds form $\Psi(F)$ is given as the 
product of the factors just described.  Note that, since $c(m)\in S_M$, it follows that  if $c(m)(x) \ne 0$ for some $x\in V(\Q)$
then $x\in M^\vee$.  Thus, all of the expressions just given for the factors of $\Psi(F)$ can be rewritten in terms of $M$, 
and, we obtain a more intrinsic version of our product formula. 

\begin{cor}\label{gen-cor} Let $M$ be an even integral lattice in $V$ and let $F$ be an $S_M$-valued weakly holomorphic form with associated 
Borcherds form $\Psi(F)$.  Let $U \subset V$ be an isotropic $2$-plane and choose a Witt decomposition (\ref{eq1}) and a $\Z$-basis $e_1$ and $e_2$ 
for $M\cap U$ with dual basis $e_1'$, $e_2'$ for $U'$. Suppose that $w$ is normalized so that $(w,e_2)=1$ and let $(w,e_1)=\tau_1$. 
Then $\Psi(F)(w)$ is the product of four terms:\hfb
(a) 
$$
\prod_{\substack{x\in M^\vee\\ \snass (x,e_2)=0\\ \snass (x,e_1)>0\\ \snass \mod M\cap \,\Q\, e_2}} \big(1-e(-(x,w))\ \big)^{c(-Q(x))(x)}.
$$
(b)
$$
\prod_{\substack{x\in M^\vee \cap U^\perp\\  \snass \mod  M\cap U\\ \snass (x,W_0)>0}} 
\bigg(\ \frac{\vartheta_1(-(x,w),\tau_1)}{\eta(\tau_1)} \, e((x,w)-\frac12 (x_U,w))^{(x,e_1')}\ \bigg)^{c(-Q(x))(x)},
$$
where $x_U = (x,e_1')e_1+(x,e_2')e_2$ is the $U$-component of $x$ and $W_0$ is a `Weyl chamber'  in $V_0(\R)$, \hfb
(c)
$$
\prod_{\substack{x\in M^\vee\cap U/ M\cap U\\ \snass x\ne 0}} 
\bigg(\ \frac{\vartheta_1(-(x,w),\tau_1)}{\eta(\tau_1)} \,  e(\frac12(x,w))^{(x,e_1')}\ \bigg)^{c(0)(x)/2}
$$
(d) and 
$$\kappa\, \eta(\tau_1)^{c(0)(0)}\,q_2^{I_0},$$
where $\kappa$ is a scalar of absolute value $1$, and 
$$I_0 = -  \sum_{m} \sum_{\substack{x\in M^\vee\cap U^\perp\\ \snass \mod M\cap U}}
 c(-m)(x)\,\s_1(m-Q(x)).$$
\end{cor} 
Here the constant $\kappa$ may differ from that in Theorem~\ref{mainthm} due to the slight shift in the factor (b). 
Notice that a nice feature of this version is that we do not need to worry about coordinates on $D$. The value $\Psi(F)(z)$ 
is simply given by evaluating on the (unique) $w$ in (\ref{Q-model}) associated to $z$  scaled so that $(w,e_2)=1$.  

\subsection{Theta translates}
Next  we would like to clarify the meaning of the, at first sight peculiar, factor which occurs together with the function $\eta(\tau_1)^{-1}\,\vartheta(-(x,w),\tau_1)$ 
in factors (b) and (c). 
We first 
recall some basic facts about theta functions, following the conventions of Mumford, \cite{mumfordAV}, Chapter 1.  
The Jacobi theta function $\vartheta_1(z,\tau)$ coincides with $\vartheta_{11}(z)$ in the classical notation 
of, say, Weber, \cite{weber}, equation (4) on p.84.  For the lattice $L_\tau = \Z\tau + \Z$, consider the alternating form 
$E(a_1\tau+b_1, a_2\tau+b_2) = a_1 b_2-a_2 b_1$ and Hermitian form $H=H_\tau$ on $\C$ given by 
$H(z_1,z_2) = z_1 v^{-1} \bar z_2$ where $\tau=u+iv$.  Define $\a_0(a\tau+b) = e(\frac12 ab)$, 
$\l_{11}(a\tau+b) = (-1)^{a+b}$ and $\a =\a_0\,\l_{11}$. 
Let $\Th(L_\tau,H_\tau,\a)$ be the corresponding  space of theta functions, i.e., the space of holomorphic functions of $z\in \C$ such that, 
for all $\ell \in L_\tau$, 
$$\theta(z+\ell) = \a(\ell) \, \exp(\pi H(z,\ell) + \frac12 \pi H(\ell,\ell))\, \theta(z).$$
This space has dimension $1$.  It is convenient and traditional to define $B(z_1,z_2) = z_1 v^{-1} z_2$ and to renormalize by setting 
$$\theta^*(z) = \exp(-\frac12\pi B(z,z))\,\theta(z).$$
Now, for $\ell = a\tau+b$,  we have 
\be\label{theta11trans}
\theta^*(z+\ell) = \l_{11}(\ell) \, e( \,a z + \frac12 a^2\tau)^{-1}\,\theta^*(z).
\ee
and we write $\Th^*(L_\tau,H_\tau,\a)$ for the corresponding space of theta functions. 
The function $\vartheta_{11}$ is then a basis vector for the space $\Th^*(L_\tau,H_\tau,\a)$. For 
example, 
(5) p.72 of \cite{weber} is precisely (\ref{theta11trans}). 

For $\eta = \eta_1\tau +\eta_2$ with $\eta_1$ and $\eta_2\in \R$ and  for $\theta\in \Th(L_\tau,H_\tau,\a)$, let
$$\theta_{\eta}(z) = \exp( - \pi H(z,\eta))\,\theta(z+\eta).$$
Then $\theta_\eta \in \Th(L_\tau, H_\tau, \a\,\gamma_\eta)$, where $\gamma_\eta: L \rightarrow \C^1$ is the character defined by 
$$\gamma_\eta(\ell) = e(E(\eta,\ell)).$$
This just amounts to the isomorphism
$$T_\eta^*\mathcal L(H_\tau,\a) \simeq \mathcal L(H_\tau, \a\,\gamma_\eta)$$
of the Proposition on p.84 of \cite{mumfordAV}. 
It is easy to check that 
$$(\theta_\eta)^*(z) = \exp(-\pi (H-B)(z,\eta) + \frac12 \pi B(\eta,\eta))\,\theta^*(z+\eta),$$
so that $(\theta_\eta)^*$ is a renormalized translate of $\theta^*$.  However,  
since the quantity $B(\eta,\eta)$ does not depend holomorphically on $\tau$, it is better to include an 
extra factor (independent of $z$) and set
\begin{align}
(\theta^{*})^{\sharp}_\eta(z) &= \exp(-\frac{\pi}2 H(\eta,\eta))\, (\theta_\eta)^*(z)\notag\\
\nass
{}&= e(\eta_1 z + \frac12\eta_1\eta)\,\theta^*(z+\eta)\label{theta-eta}\\
\nass
{}&= \a_0(\eta)\, e(\eta_1 z + \frac12\eta_1^2 \tau)\,\theta^*(z+\eta).\notag
\end{align}
Then $(\theta^{*})^{\sharp}_\eta$ is again a basis for the space $\Th^*(L_\tau, H_\tau, \a\,\gamma_\eta)$.

In view of these remarks, we may write the  expression occurring in the product in (ii) of Theorem~\ref{mainthm} as a 
normalized translate by $\eta=-\L_2 = -\l_{21}\tau_1+\l_{22}$.  More precisely, setting
$$\Theta_1[\eta](z,\tau) = (\vartheta_{11})^{\sharp}_\eta(z),$$
for convenience, and inspecting expression (\ref{theta-eta}) for $(\theta^{*})^{\sharp}_\eta$, we have
$$\vartheta_1(-(x_0,w_0)-\L_2,\tau_1) \,e((x_0,w_0)\l_{21} + \L_2\l_{21}) =\Theta_1[-\L_2](-(x_0,w_0),\tau_1).$$

In the general case of Corollary~\ref{gen-cor},  we have a identifications, 
$$U(\R) \isoarrow \C, \quad\text{and}\quad  U(\R)/M\cap U \isoarrow \C/L_{\tau_1} = E_{\tau_1}, \qquad u\mapsto (u,w).$$
If $x\in M^\vee\cap U^\perp$, then the point $(x_U,w)$ attached to the $U$-component of $x$ determines a torsion point of $E_{\tau_1}$.
Writing
$$(x,w) = (x_0,w_0) + (x_U,w), $$
we have the expression
$$\vartheta_1(-(x,w),\tau_1) \, e((x,w)-\frac12 (x_U,w))^{(x,e_1')} = \Theta_1[ -(x_U,w)](-(x_0,w_0), \tau_1).$$
in the factors in (b) and (c). 
In particular, the factors 
$$\vartheta_1(-(x,w),\tau_1) \, e((x,w)-\frac12 (x_U,w))^{(x,e_1')} = \Theta_1[ -(x_U,w)](0, \tau_1)$$
occurring in (c) are thetanullwerte.
An easy computation using (\ref{theta-eta}) shows that, for $\ell\in L_{\tau_1}$, 
$$\Theta_1[\eta+\ell](z,\tau_1) = \a(\ell)\,e(\frac12 E(\eta,\ell))\,\Theta_1[\eta](z,\tau_1).$$
Thus, if the coset representatives in (b) and (c) are changed by elements of $M\cap U$, the theta translates are changed by certain 
roots of unity. 

\subsection{Local Borcherds products}\label{LBP}  In \cite{bruinier.freitag}, Bruinier and Freitag considered the local Picard group in a neighborhood of 
a generic point of a rational $1$-dimensional boundary component associated to an isotropic $2$-plane $U$.  
In particular, they introduced local Borcherds products 
attached to vectors $x\in M^\vee\cap U^\perp$, Definition~4.2, p.16. In our notation, such a product is given by 
$$\Psi_x(w) = (1- e((x,w))\,\prod_{a>0} (1- q_1^a \,e((x,w)))\, (1- q_1^a \,e(-(x,w))).$$
On the other hand, by the classical product formula (\ref{jacobi-II}), we have
$$\frac{\theta_1(-(x,w),\tau_1)}{\eta(\tau_1)} = -i q_1^{\frac1{12}}\,e(-\frac12(x,w))\,\Psi_x(w),$$
so that factor (b) in Corollary~\ref{gen-cor} is essentially a product of such local Borcherds products. 
Of course, this factor accounts for the divisor of $\Psi(F)$ in a neighborhood of the boundary component. 

\section{Fourier-Jacobi expansions}

In this section, we make explicit the information about the Fourier-Jacobi expansion of $\Psi(F)$ that is contained in our product formula as given
 in Corollary~\ref{gen-cor}.
For simplicity we normalize $\Psi(F)$ so that 
$\kappa=1$, for our fixed $U$,  and write the Fourier-Jacobi expansion as
\be\label{def-FJ}
\Psi(F)(w) = q_2^{I_0}\, \sum_{k\ge 0} \Psi_k(\tau_1,w_0)\, q_2^{k}.
\ee
Then the leading coefficient is
the product of the factors in (b), (c) and (d), with the power of $q_2$ in (d) omitted:
\begin{multline*}
\Psi_0(\tau_1,w_0) = \eta(\tau_1)^{c(0)(0)}\, \prod_{\substack{x\in M^\vee\cap U/ M\cap U\\ \snass x\ne 0}} 
\bigg(\ \frac{\vartheta_1(-(x,w),\tau_1)}{\eta(\tau_1)} \,  e(\frac12(x,w))^{(x,e_1')}\ \bigg)^{c(0)(x)/2}\\
\nass
\nass
\times \prod_{\substack{x\in M^\vee \cap U^\perp\\  \snass \mod  M\cap U\\ \snass (x,W_0)>0}} 
\bigg(\ \frac{\vartheta_1(-(x,w),\tau_1)}{\eta(\tau_1)} \, e((x,w)-\frac12 (x_U,w))^{(x,e_1')}\ \bigg)^{c(-Q(x))(x)}
\end{multline*}
Note that, in the product on the first line of this formula, $(x,w)$ does not depend on $\tau_2'$ or $w_0$; 
it has the form $\a \tau_1 + \b$ for $\a$, $\b\in \Q$ and hence $\vartheta_1(-(x,w),\tau_1)$ is a division point value of the 
Jacobi theta function.  The second line of the product gives the dependence on $w_0$.  We will make this more 
explicit in a moment. 

To compute more Fourier-Jacobi coefficients, we consider the product in (a) of Corollary~\ref{gen-cor}.
which we write in the form
$$\exp\bigg( -\sum_{x} c(-Q(x))(x) \sum_{n=1}^\infty \frac{1}{n} e(-n(x,w))\ \bigg).$$
Here $x$ runs over the same index set as in (a).
Note that, since $x\in M^\vee$ and $e_1$ is a primitive vector in $M$, we have $(x,e_1)=a\in \Z_{>0}$, and we can write
$$x = a e_1'+ \dot x$$
where $\dot x \in U^\perp$.  If $e_1'\in M^\vee$, then $\dot x\in M^\vee\cap U^\perp$, but this need not always be the case. In any case, the set of 
components $\dot x$ arising for $x\in M^\vee$ with $(x,e_2)=0$ is a union of $M\cap U$ cosets. 
We can write the product (a) as
$$
\exp\bigg(  -\sum_{n=1}^\infty \frac{1}{n}  \sum_{a=1}^{\infty} q_2^{an} \sum_{\substack{\dot x\in U^\perp\\ \snass \mod M\cap U}} c(-Q(x))(x)\
e(-n(\dot x, w))\ \bigg),
$$
where $x = \dot x+ a e_1'$. 
Note that in this expression, we are evaluating $c(m)\in S_M\subset S(V(\A_f))$, for $m= -Q(\dot x+a e_1')$, on the vector $\dot x+ a e_1'$, 
and hence are imposing, in particular, the condition that $\dot x+ a e_1'\in M^\vee$. 
Thus, we obtain the following
striking formula. 
\begin{cor}
\be\label{GNformula}
\Psi(F)(w) = q_2^{I_0} \,\Psi_0(\tau_1,w_0)\,\exp\bigg( -\sum_{n=1}^\infty \frac{1}{n}  \sum_{a=1}^{\infty} q_2^{an}\,\Theta_{a,n}(F)(\tau_1,w_0)\ \bigg).
\ee
where
\be\label{FtoJacobi}
\Theta_{a,n}(F)(\tau_1,w_0) = \sum_{\substack{\dot x\in U^\perp\\ \snass \mod M\cap U}} c(-Q(x))(x)\
e(-n(\dot x, w)).
\ee
Here $\dot x + a e_1'$. 
\end{cor}
Formulas of this sort occur frequently in the work of Gritsenko, \cite{gritsenko.AG}, Gritsenko-Nikulin, \cite{grit.nik, GK-I, GK-II}, 
Cl\'ery-Gritsenko, \cite{clery.gritsenko},  and others. 
Indeed, in these papers, (\ref{GNformula}) is essentially taken as the definition of a lift from suitable space of Jacobi forms to 
modular forms for $\text{\rm O}(n,2)$, and the modularity is proved by using information about generators for the
group $\Gamma_L$, as was was the case in the original paper of Borcherds, \cite{borch95}. 
Here we obtain these expansion form the regularized theta lift defined in Borcherds second paper, 
\cite{borch98} and hence we see that every Borcherds lift $\Psi(F)$ from that paper has such an expression. 

Expanding the exponential series, we obtain expressions for the Fourier-Jacobi coefficients of $\Psi(F)$. 
\begin{cor}  Writing $\Psi_k = \Psi_k(\tau_1,w_0)$ and $\Theta_{a,n} = \Theta_{a,n}(F)(\tau_1,w_0)$, 
\begin{align}\label{FJ.polys}
\Psi_1/\Psi_0 &=  \Theta_{1,1},\notag\\
\nass
\Psi_2/\Psi_0 & = -\Theta_{2,1} - \frac12\, \Theta_{1,2} +\frac12 \Theta_{1,1}^2 ,\\
\nass
\Psi_3/\Psi_0 & = -\Theta_{3,1} -\frac13 \Theta_{1,3} + \Theta_{1,1}\,\Theta_{2,1} +\frac12 \Theta_{1,1}\,\Theta_{1,2} - \frac{1}{3!} \Theta_{1,1}^3,\notag\\
\nass
\Psi_4/\Psi_0&= -\Theta_{4,1} -\frac12 \Theta_{2,2} -\frac14 \Theta_{1,4}+ \dots + \frac{1}{4!} \Theta_{1,1}^4\notag\\
\dots & \dots\notag
\end{align}
\end{cor}

To make these series more explicit, we choose $L\subset M$ as in (1.10) and write 
$$c(m) = \sum_{\l\in L^\vee/L} c_\l(m) \,\ph_\l,$$
as in (2.2). Recall that $M\cap U = L\cap U$. 
For $x= \dot x + a e_1'$, we have $\ph_\l(x) \ne 0$ implies that $\l_{12}=0$ and $\l_{11} \equiv a\mod N$. Then we write
\begin{align*}
x&= \dot x + a e_1'
= (\l_{21}-b)e_1 + (\l_{22}-c) e_2 + x_0 + a e_1'\\
\noalign{so that}
m&= -Q(x) = -Q(\dot x) -a (e_1',\dot x)
= -Q(x_0) + ab - a\l_{21}\\
\noalign{where $b$ and $c\in \Z$. 
Hence $a\mid (m+Q(x_0) +a\l_{21})$. Also}
\nass
(\dot x,w) &= (x_0,w_0) + (\l_{21}-b) \tau_1 +\l_{22} -c.
\end{align*}
With this notation,  we can write (\ref{FtoJacobi}) as
$$\Theta_{a,n}(F)(\tau_1,w_0)=\sum_{\substack{\l \\ \l_{12}=0\\ \snass \l_{11} \equiv a \text{ mod }(N)}} \sum_m c_\l(m)\,q_1^{a^{-1}nm}
\sum_{\substack{x_0\in \l_0 + L_0\\ \snass a\mid (m+Q(x_0) +a\l_{21})}}  q_1^{a^{-1}nQ(x_0)}\,e(-n(x_0,w_0)-n\L_2).  $$
where $\L_2= \l_{21}\tau_1+\l_{22}$. 
For $a=1$ and $n=1$, this is simply
$$\Theta_{1,1}(F)(\tau_1,w_0)=\sum_{\substack{\l \\ \l_{12}=0\\ \snass \l_{11} \equiv 1 \text{ mod }(N)}} \sum_m c_\l(m)\,q_1^{m}
\sum_{\substack{x_0\in \l_0 + L_0}}  q_1^{Q(x_0)}\,e(-(x_0,w_0)-\L_2).  $$
Here the divisibility condition in the inner sum has been dropped. Indeed, if 
$x\in \l+L$, we have $Q(x) \equiv Q(\l) \mod \Z$, and if  $c_\l(m)\ne 0$ we have $m+Q(\l)\in \Z$. Thus, for $x$ 
as above with $a=1$, $m + Q(x_0) + \l_{21} \in \Z$.

The following transformation law is not difficult to check. 
\begin{lem}
Assume that $L_0$ is even integral, and for $b_1$ and $b_2\in L_0$, let $\L_b= b_1\tau_1+b_2$. 
Then 
$$\Theta_{a,n}(F)(\tau_1,w_0+\L_b) = e(-a n Q(b_1)\tau_1 -an (w_0,b_1))\, \Theta_{a,n}(F)(\tau_1,w_0) .$$
\end{lem} 
\begin{proof}
Noting that $(\dot x, b_i) =(x_0,b_i)\in \Z$, since $x_0\in L_0^\vee$ and $b_i\in L_0$, we can write
\begin{multline*}
q_1^{a^{-1}nQ(x_0)}\,e(-n(x_0,w_0+\L_b)-n\L_2)\\
\nass
 = q_1^{a^{-1}n Q(x_0-ab_1) -a n Q(b_1)}\,e(-n(x_0 - ab_1,w_0)-n\L_2)\,e(-a n(b_1,w_0)).
\end{multline*} 
so that all summands scale in the same way.  
\end{proof}
We will omit the transformation law under $\SL_2(\Z)$ and simply note that the weight of $F$ is $1-\frac{n}2$ and that of the theta 
function associated to $L_0$ is $\frac{n}2-1$, so that $\Theta_{a,n}(F)$ is a generalized 
(weak) Jacobi form of weight $0$ and index $an$, cf. for example, \cite{gritsenko.AG}, \cite{clery.gritsenko}.

Finally, with the same notation, we can write 
\begin{multline*}
\Psi_0(\tau_1,w_0) = \eta(\tau_1)^{c_0(0)}\, \prod'_{\substack{\l_{21}, \l_{22}\in N^{-1}\Z/\Z}} 
\bigg(\ \frac{\vartheta_1(-\L_2,\tau_1)}{\eta(\tau_1)} \,  e(\frac12\,\L_2\,\l_{21})\ \bigg)^{c_\l(0)/2}\\
\nass
\nass
\times \prod_{\substack{x\in M^\vee \cap U^\perp\\  \snass \mod  M\cap U\\ \snass (x,W_0)>0}} 
\bigg(\ \frac{\vartheta_1(-(x,w),\tau_1)}{\eta(\tau_1)} \, e((x,w)-\frac12 (x_U,w))^{(x,e_1')}\ \bigg)^{c(-Q(x))(x)}
\end{multline*}
Here, in the first line, $\L_2 =\l_{21}\tau_1+\l_{22}$ and the prime indicates that $\l_{21}$ and $\l_{22}$ are not both zero. 

\section{A computation of the regularized integral}

\subsection{Passage to a mixed model} To obtain his product formulas, Borcherds computes that Fourier expansion of the regularized theta lift 
along the maximal parabolic which is the stabilizer of an isotropic line in $V$.  We compute, instead, the 
expansion with respect to the maximal parabolic $G_U$ stabilizing the isotropic $2$-plane $U$. 
To do this, we switch to a model of the Weil representation associated to a polarization arising from $U$.

Let $W = X+Y$, $\gs{}{}$, be the standard $2$ dimensional symplectic vector space with polarization.  
Choosing basis vectors $e_X$ for $X$ and $e_Y$ for $Y$
with $\gs{e_X}{e_Y}=1$,
we have 
$W(\Q)= \Q^2$ (row vectors) with the right action of $\Sp(W) = \SL_2(\Q)$. 
The symplectic vector space $V\tt W$, $(\ ,\ )\tt \gs{}{}$ has two polarizations 
$$V\tt W = V\tt X + V \tt Y = \big(\ V_0 \tt X + U' \tt W\, \big) + \big( \ V_0 \tt Y + U\tt W\, \big).$$
For the first of these, we have the standard Schr\"odinger model of the Weil representation on $S(V\tt X(\A)) = S(V(\A))$, the Schwartz space of $V(\A)$. 
For the second, we have a mixed model of the Weil representation on the space
$S( (V_0 \tt X + U' \tt W)(\A))$.
We change model of the Weil representation using a partial Fourier transform.  We write $\ph\in S(V(\A))$ as a function of $(x_0,x_1,x_2)$ 
where $x_0\in V_0\tt X(\A) = V_0(\A)$, $x_1\in U'\tt X(\A) = U'(\A) = \A^2$ and $x_2\in U\tt X(\A) = U(\A)= \A^2$, via our choice of bases. 
Then define
$$ S(V(\A)) \isoarrow S(V_0(\A)) \tt S( U'\tt W(\A)),\qquad \ph \mapsto \hat\ph,$$
$$\hat\ph(x_0,x_1,\eta_2) = \int_{\A^2} \ph(x_0,x_1,x_2)\,\psi(x_2\dop \eta_2)\,dx_2,$$
where we take $\psi$ to be the standard additive character of $\A/\Q$ that is trivial on $\hat\Z$ and restricts to $x\mapsto e(x)$ on $\R$. 
Here $\eta_2\in U'\tt Y(\A) = U'(\A) = \A^2$, and the natural pairing of $U\tt X$ and $U'\tt Y$, defined by the restriction of $(\ ,\ )\tt \gs{}{}$,  becomes the dot product
on $\A^2$. 
For an element $g' \in G'_\A$, we have
$$\widehat{\o(g')\ph}(x_0,\eta_1,\eta_2) = \o_0(g') \hat{\ph}(x_0,[\eta_1,\eta_2]g'),$$
where $\o_0$ is the Weil representation for $V_0$. 
Similarly, for an element of the Levi factor of $P_U$, we have 
$$\widehat{m(\a,h)\ph}(x_0,\eta) = \widehat{\ph}(h^{-1} x_0, {}^t \a\,\eta).$$
We will view the argument $\eta = [\eta_1,\eta_2]$ 
as an element of 
$$U'\tt W(\A) = \Hom(U,W)(\A) = M_2(\A).$$ Note that, under this transformation there is an identity of theta distributions
$$\sum_{x\in V(\Q)} \ph(x) = \Theta(\ph) =\hat\Theta(\hat \ph) = \sum_{x_0\in V_0(\Q), \eta\in M_2(\Q)} \hat\ph(x_0,\eta).$$

Since the regularized theta lift involves an integral over $\Gamma'\back \H$, we
decompose according to $\Gamma'$-orbits:
\begin{align*}
\theta(g',\ph) & = \sum_{\substack{x_0\in V_0(\Q) \\ \snass \eta\in M_2(\Q)}} \o_0(g')\hat{\ph}(x_0,\eta g')\\
\nass
{}&= \sum_{\eta /\sim} \ \sum_{\gamma\in \Gamma'_\eta\back \Gamma'} \theta_\eta(\gamma g',\ph),
\end{align*}
where 
$$\theta_\eta(g',\ph) = \sum_{x_0\in V_0(\Q)} \o_0(g')\hat{\ph}(x_0,\eta g').$$

A set of orbit representatives for $\SL_2(\Z)$ acting on $M_2(\Q)$ by right multiplication is given by:
\begin{align}
\label{etareps}
&0, \quad \bpm 0& a\\0&b\epm, \ \text{$a>0$, or $a=0$, $b>0$,  in $\Q$}, \\
\nass
\quad &\bpm a & b \\ 0 & \a\epm \ \text{$a$, $\a\in \Q^\times$, $a>0$, $b\in \Q\!\mod a\Z$.}\notag
\end{align}
As stabilizers, we have $\SL_2(\Z)$, $\{ \bpm 1& n\\ {}&1\epm \mid n\in \Z\}$,  and $1$ respectively, and we write $\Gamma'_\eta$ for their inverse images in 
$\Gamma'$.

Thus, we get a decomposition 
\be\label{theta.decompo-I}
\bgs{F(g'_\tau)}{\theta(g'_\tau,\ph_\infty)}=  
\sum_{\eta /\sim}\  \sum_{\gamma\in \Gamma'_\eta\back \Gamma'} 
\bgs{F(g'_{\gamma(\tau)})}{\theta_\eta(g'_{\gamma(\tau)},z)}.
\ee
Note that, for the terms with $\eta\ne0$, the contributions of $\gamma$ and $-\gamma$ are identical since $-1_2$ acts trivially on $\H$. 
This will result in a factor of $2$ for such terms when we unfold.

We will apply this identity to functions of the form $\ph_{\tau,z}\tt \ph$ for $\ph \in S(V(\A_f))$ and
\be\label{gaussian.def}
\ph_{\tau,z}(x)  = \o(g'_\tau)\ph_{\infty}(x,z) = v^{\frac{n+2}4} \,e(\tau Q(x))\, \exp(-2\pi v R(x,z)).
\ee

Now we return to the decomposition 
\be\label{theta.decompo-I}
\bgs{F(g'_\tau)}{\theta(g'_\tau,\ph_\infty(z))}=  
\sum_{\eta /\sim}\  \sum_{\gamma\in \Gamma'_\eta\back \Gamma'} 
\bgs{F(g'_{\gamma(\tau)})}{\theta_\eta(g'_{\gamma(\tau)},z)}.
\ee
and we break this into three blocks according to the rank of $\eta$:
\be\label{theta.decompo-II}
\bgs{F(g'_\tau)}{\theta(g'_\tau,\ph_\infty(z))}=  
\sum_{i=0}^2 \sum_{\substack{\eta /\sim\\ \text{rank}(\eta)=i}}\  \sum_{\gamma\in \Gamma'_\eta\back \Gamma'} 
\bgs{F(g'_{\gamma(\tau)})}{\theta_\eta(g'_{\gamma(\tau)},z)}.
\ee
Note that each block defines a $\Gamma'$-invariant function on $\H$. 
Moreover, for our choice of representatives, all $\eta$ of a given rank have the same stabilizer $\Gamma'_\eta$  in 
$\Gamma'$.  We obtain a corresponding decomposition of the regularized theta integral (\ref{reg-int})
$$\P(z;F) = \sum_{i=0}^2 \P_i(z;F),$$
where
$$\P_i(z;F) = \int_{\Gamma'\back \H}^{\text{reg}} \sum_{\substack{\eta /\sim\\ \text{rank}(\eta)=i}}\  \sum_{\gamma\in \Gamma'_\eta\back \Gamma'} 
\bgs{F(g'_{\gamma(\tau)})}{\theta_\eta(g'_{\gamma(\tau)},z)}\, v^{-2}\,du\,dv.
$$

The case $i=0$, where $\eta=0$,  was essentially already treated by Borcherds, \cite{borch98}, and we will review the result in
section~\ref{sec.zero.orbit}  below. 

For $i=1$ and $2$, we need to show that  
$$\phi_i(s,z) = \lim_{T\to \infty} \int_{\Cal F_T} \sum_{\substack{\eta /\sim\\ \text{rank}(\eta)=i}}\  \sum_{\gamma\in \Gamma'_\eta\back \Gamma'} 
\bgs{F(g'_{\gamma(\tau)})}{\theta_\eta(g'_{\gamma(\tau)},z)}\, v^{-s-2}\,du\,dv
$$
defines a holomorphic function of $s$ in a right half plane, to prove analytic continuation to a neighborhood of $s=0$,  and to compute the constant term there.

\subsection{Non-singular terms}
We will need to restrict $z$ to a certain open subset $D^o$ of $D$.  To describe it, we need to introduce some basic constants. 
It is a standard fact that  the Fourier coefficients of a weakly holomorphic modular form have sub-exponential growth, i.e.,
there is a positive 
constant $c_F$, depending on $F$, such that 
for large $m$, 
$$|c_\l(m)| = O(e^{2\pi c_F \sqrt{m}}),$$
for all $\l\in L^\vee/L$. 
The Fourier coefficients $c(m)$ of $F$ lie in $S(V(\A_f))$, and we write $\hat c(m)$ for their images under the 
partial Fourier transform
$$\hat c(m)(x_0,\eta) = \int_{\A_f^2} c(m)(x_0,\eta_1,x_2)\,\psi(x_2\cdot \eta_2)\,dx_2.$$
Let 
$B_a$ (resp. $B_\a$) be a lower bound for the set of $a$, $a>0$,  (resp. $|\a|$) occurring as components of 
a rank $2$ orbit representative $\eta$ for which 
$\hat c(m)(\cdot , \eta) \ne0$ for some $m$. Finally, let $B_m$ be an upper bound on the set of $m>0$ for which $c(-m)\ne 0$. 

By some tedious estimates, which we omit, we obtain the following. 
\begin{lem}\label{lem-tedious}  Suppose that $z$ lies in the region $D^o$ in $D$ where 
$$v_2  > \max\left( \frac{8 B_m}{B_a^2} v_1, \frac32 \frac{c_F^2}{B_\a^2}\, v_1^{-1}\right).$$
Then $\phi_2(s,z)$ defines an entire function of $s$. Moreover, its value at $s=0$ can be computed by unfolding and 
is given by  
$$\phi_2(0,z)= 2\sum_{\eta/\sim} \sum_m\sum_{x_0\in V_0(\Q)} \int_\H 
\big(\ \hat c(m)\cdot \widehat{\ph_{\tau,z}}\ \big)\big (x_0, \eta)\,q^m\,\,v^{-\ell/2-2}\,du\,dv.$$
\end{lem}

Here note that 
$$\hat c(m)\cdot \widehat{\ph_{\tau,z}}  \in S(V_0(\A))\tt S(M_2(\A)).$$

The first step is to determine $\widehat{\ph_{\tau,z}}$. 
The majorant can be expressed as follows.
\begin{lem} (i) 
$$R(x,z) = 2|(w,\bar w)|^{-1}\,|(x,w)|^2.$$
(ii) 
$$(x,w) = \big( -{}^tx_1 (\tau_2 J +Q(v_0)) + (x_0,v_0) + {}^tx_2\, \big) \bpm \tau_1\\1\epm.$$
Here the expression in the first factor on the right side is a row vector. 
\hfill\break
(iii) 
$$|(x,w)|^2 = \big\vert \ {}^t(x_2 - B)\bpm \tau_1\\1\epm\ \big\vert^2,$$
where
$$B = (Q(v_0) -\tau_2 J)\, x_1 - (v_0,x_0).$$
\end{lem}

Using these expressions and a straightforward computation of the partial Fourier transform, we obtain the following. 

\begin{lem}  Write $\eta = [\eta_1,\eta_2]\in M_2(\R)$ and let $\eta_\tau = \eta\,\bpm \tau\\ 1\epm = \tau\eta_1+\eta_2$. 
Then  
$$\widehat{\ph_{\tau,z}}(x_0,\eta) = v^{\frac{n-2}4}\,v_2 \, e(Q(x_0) \tau)\,
e( B\dop \eta_\tau) \, \exp(-\pi v_2 v^{-1} v_1^{-1}  \, \big|\phantom{\bigg|}(1, -\tau_1)\,\eta_\tau\big|^2),$$
where $B = (Q(v_0) -\tau_2 J)\, \eta_1 - (v_0,x_0)$.
\end{lem}

\begin{lem} Suppose that 
\be\label{positivecond} 
m+ Q(x_0-a v_{01})  + a^2 v_1^{-1} v_2> 0 .
\ee
Then the value of the integral
$$
\int_{\H}  \widehat{\ph_{\tau,z}}\big (x_0, \bpm a&b\\{}&\a\epm)\,q^m\,\,v^{-s-\ell/2-2}\,du\,dv
$$
at $s=0$ 
is 
$$a^{-1}|\a|^{-1}e\big(\, \a  \,\big(\, a \tau_2' - (x_0, w_0) + a^{-1}(m+Q(x_0))\tau_1\big)\,\big) \,e( - a^{-1} b (m+Q(x_0))\,), $$
if $a\a>0$ and  
$$a^{-1}|\a|^{-1}e\big(\, \a  \,\big(\, a \bar\tau_2' - (x_0,\bar w_0) + a^{-1}(m+Q(x_0))\bar\tau_1\big)\,\big) \,e( - a^{-1} b (m+Q(x_0))\,),$$
if $a\a<0$.
\end{lem}
\begin{proof}
In the integrand here 
$$B\dop \eta_\tau= (\,a\, Q(v_{01})-(v_{01},x_0)\ )\,(a\tau+b) + (\, a\,\frac12 (v_{02},v_{01}) +a\,\tau_2 - (v_{02},x_0)\ ) \,\a,$$
so that 
$$\widehat{\ph_{\tau,z}}(x_0, \eta)\,q^m\,\,v^{-s-\ell/2-2}
= v^{-s-2}\,v_2 \, e(\bold C \tau + \bold B \a+ \bold B')\,
\exp\bigg( - \pi v^{-1}{ {v_1}^{-1}} \,v_2\, |a\tau+b-\a\tau_1|^2 \ \bigg),$$
where, for simplicity, we let
\begin{align*}
\bold C& = m+ Q(x_0-a v_{01}),\\ 
\bold B& = a\,\frac12 (v_{02},v_{01}) +a\,\tau_2 - (v_{02},x_0),\qquad \text{and}\quad \bold B' = (\,a\, Q(v_{01})-(v_{01},x_0)\, )\,b.
\end{align*}

We first compute the integral over $\R$ with respect to $u$ to obtain
\begin{multline*}
 v^{-s-2}\,v_2 \, e(\bold B \a+\bold B')\, \exp( -2\pi \bold C v -\pi v^{-1}v_1^{-1} v_2 (av-\a v_1)^2)\\
\nass
{}\times e(\bold C a^{-1}(\a u_1-b))\,(v^{-1} v_1^{-1} v_2)^{-\frac12}\,a^{-1} \exp(-\pi (v^{-1} v_1^{-1} v_2)^{-1} \,a^{-2}\bold C^2).
\end{multline*}

Next we have to compute the integral over $(0,\infty)$ with respect to $v$. First we pull out the 
factor
$$a^{-1}\, v_2 \,( v_1^{-1} v_2)^{-\frac12}\, e(\bold B \a+\bold B')\, \exp( 2\pi a\a v_2)\, e(\bold Ca^{-1}( \a u_1-b)),$$
which has no dependence on $v$. The integral of the remaining factor is 
\begin{align*}
&\int_0^\infty v^{-s-\frac32}\, \exp(-\pi v v_1 v_2^{-1}\,(\, a^{-1}\,\bold C + a\,v_1^{-1} v_2\,)^2  - \pi v^{-1} v_1v_2\a^2) \,dv\\
\nass
{}&= 2 \left(\frac{|\bold C+a^2\,v_1^{-1} v_2|}{v_2 |a||\a|}\right)^{s+\frac12} K_{-s-\frac12}(2\pi v_1|\a| |a|^{-1}\,|\bold C+a^2\,v_1^{-1} v_2|),
\end{align*}
using the formula
$$\int_0^\infty v^{\nu-1} \, \exp( -a v -b v^{-1}) \,dv = 2 \left(\frac{a}{b}\right)^{-\frac{\nu}2}K_\nu(2\sqrt{ab}).$$

Collecting terms, we have
\begin{multline*}
a^{-1}\,(v_1v_2)^{\frac12} \, e(\bold B \a +\bold B'-i a\,\a v_2 + \bold C a^{-1}( \a u_1-b))\,\\
\nass
\times 2 \left(\frac{|\bold C+a^2\,v_1^{-1} v_2|}{v_2|a| |\a|}\right)^{s+\frac12} K_{-s-\frac12}(2\pi v_1|\a| |a|^{-1}\,|\bold C+a^2\,v_1^{-1} v_2|).
\end{multline*}
Next setting $s=0$,  and recalling that 
$$K_{-\frac12}(2\pi r) = \frac12\,r^{-\frac12} e^{-2\pi r},$$
and simplifying, we have
\be
a^{-1}\,|\a|^{-1}\,e\big(\bold B \a+ \bold B'-i a\,\a v_2 + \bold C \,a^{-1}(\a u_1-b)+i v_1|\a| |a|^{-1}\,|\bold C+a^2\,v_1^{-1} v_2|\big).
\ee

Suppose that 
\be\label{Cpositive}
\bold C+ a^2 v_1^{-1} v_2  = m+Q(x_0-av_{01}) + a^2v_1^{-1} v_2 >0.
\ee
Then we have 
\begin{multline*}
a^{-1}|\a|^{-1}\,e(\bold B \a+\bold B'-i a\,\a v_2 + \bold C \,a^{-1}(\a u_1-b)+i v_1|a|^{-1}|\a| \bold C+i |a||\a| v_2)\\
\nass
{} = \begin{cases}a^{-1} |\a|^{-1}\,e(\bold B \a +\bold B'+ \bold Ca^{-1}( \a \tau_1-b))
&\text{if $a\a>0$,}\\ 
\nass
a^{-1}|\a|^{-1}\,e(\bold B \a+\bold B'-2i a\a v_2 + \bold C a^{-1}(\a \bar\tau_1-b))&\text{if $a\a<0$.}
\end{cases}
\end{multline*}
Rewriting in terms of holomorphic coordinates, we obtain the claimed expressions. \end{proof}

We now suppose that the lattice $L$ is chosen as in section~\ref{subsection-lattice} above. Coset representatives $\l\in L^\vee/L$ 
then have the form $\l = \l_0+ \l_1+\l_2$ with $\l_0\in L_0^\vee$, $\l_1\in (\Z/N\Z)^2$ and $\l_2\in (N^{-1}\Z/\Z)^2$. 
Then 
\be\label{hatcm}
\hat c(m) = \sum_\l c_\l(m) \,\hat{\ph}_\l
\ee
and an easy computation shows that 
\be\label{hatph}
\hat \ph_\l(x_0,\eta_1,\eta_2) = e(-\l_2\dop \eta_2)\, \ph_{\l_0}(x_0)\,\ph_{\l_1}(\eta_1) \,\ph_{\hat \Z^2}(\eta_2).
\ee
For our orbit representative $\eta$, this will vanish unless $\a\in \Z$, $a\in \l_{11}+N\Z$, $\l_{12}=0$,  and $b\in \Z$, in  which case it has the value
$$e(-\l_{21} b- \l_{22}\a)\,\ph_{\l_0}(x_0).$$ 
For fixed $\eta$ with  $\a>0$, we have
\begin{multline}
\sum_m  \sum_{x_0\in V_0(\Q)}\int_{\H} \big(\ \hat c(m)\cdot \widehat{\ph_{\tau,z}}\ \big)\big (x_0, \eta)\,q^m\,\,v^{-s-\ell/2-2}\,du\,dv\,\vert_{s=0}\\
\nass
{}=\sum_{\substack{\l\\ \snass \l_{12}=0}} \sum_m  c_\l(m)\sum_{x_0\in V_0(\Q)} a^{-1}|\a|^{-1}e\big(\, \a  \,\big(\, a \tau_2' - (x_0, w_0) + a^{-1}(m+Q(x_0))\tau_1\big)\,\big) \,\\
\nass
{}\times e( - a^{-1} b (m+Q(x_0))\,)\,e(-\l_{21} b- \l_{22}\a)\,\ph_{\l_0}(x_0)\,\ph_{\l_{11}}(a)\,\ph_{\Z}(\a).
\end{multline}
Now the transformation properties of $F$ imply that $m+Q(x_0) +\l_{21}a \in \Z$, 
for the terms occurring in this sum\footnote{i.e., for $\bpm a\\0\epm \in \l_1+N\Z^2$ and $x_0\in \l_0+L_0$, we have
$c_\l(m) \ne 0$ implies that $m+Q(\l_0) + \l_1\dop \l_2 \in \Z$.}. 
Taking the sum on $b$ modulo $a\,\Z$, we obtain 
\begin{multline*}
\sum_{\substack{\l\\ \snass \l_{12}=0}} \sum_m  c_\l(m)\sum_{\substack{x_0\in V_0(\Q)\\ \snass
a\mid (m+Q(x_0)+a\l_{21})}} |\a|^{-1}e\big(\, \a  \,\big(\, a \tau_2' - (x_0, w_0) + a^{-1}(m+Q(x_0))\tau_1\big)-\a \l_{22}\,\big) \,\\
\nass
{}\times \ph_{\l_0}(x_0)\,\ph_{\l_{11}}(a)\,\ph_{\Z}(\a).
\end{multline*}
The analogous contribution for $\a<0$ is the same except that $\tau_2'$ and $\tau_1$ are replaced by $\bar \tau_2'$ and $\bar \tau_1$ respectively. 
Now the sum on $\a>0$ yields
\begin{multline}\label{alpha-pos}
-\sum_{\substack{\l\\ \snass \l_{12}=0}} \sum_m  c_\l(m)\sum_{\substack{x_0\in V_0(\Q)\\ \snass
a\mid (m+Q(x_0)+a\l_{21})}} \log(1-e\big( \,\big(\, a \tau_2' - (x_0, w_0) + a^{-1}(m+Q(x_0))\tau_1\big)-\l_{22}\,\big) \,)\\
\nass
{}\times \ph_{\l_0}(x_0)\,\ph_{\l_{11}}(a),
\end{multline}
while the analogous sum for $\a<0$ yields its complex conjugate. 

Thus, the whole contribution will be
\begin{multline}\label{alpha-whole}
-2\sum_{\substack{\l\\ \snass \l_{12}=0}} \sum_m  c_\l(m)\sum_{a}\sum_{\substack{x_0\in V_0(\Q)\\ \snass
a\mid (m+Q(x_0)+a\l_{21})}} \log|1-e\big( \,\big(\, a \tau_2' - (x_0, w_0) + a^{-1}(m+Q(x_0))\tau_1\big)-\l_{22}\,\big) \,|^2\\
\nass
{}\times \ph_{\l_0}(x_0)\,\ph_{\l_{11}}(a).
\end{multline}
Note that, as remarked before,  the factor of $2$ arises since, in the unfolding, $\gamma$ and $-\gamma$ make the same contribution. 

Now (\ref{alpha-whole}) is $-2\log|\ |^2$ of the product  
\be\label{ns-prod}
\prod_{\substack{\l\\ \snass \l_{12}=0}}\prod_m \bigg( \prod_{\substack{a\in \l_{11}+N\Z\\ \snass a>0}}
\prod_{\substack{x_0\in \l_0+L_0\\ \snass
a\mid (m+Q(x_0)+a\l_{21})}} 
\big(1-e(a \tau_2' - (x_0, w_0) + a^{-1}(m+Q(x_0))\tau_1-\l_{22}\big)\bigg)^{c_\l(m)}.
\ee
It is easy to check that no factor in this product can vanish in the region
$$v_2 = v_2' -  Q(v_{01})\, v_1  >  B_m\, v_1.$$
It is also not difficult to check the absolute (uniform) convergence of this product in a region of the
form  
$$ v_2 >  B_m\, v_1 + c_F^2 v_1^{-1}.$$

It is interesting to remark that, in this calculation the conjugate pair of factors arise naturally for each $x_0$ and there is no choice of Weyl chamber
involved. This is consistent with the fact that the expansion we are computing is associated to a $1$-dimensional boundary component where 
no choice of rational polyhedral cone is being made. In contrast, the formulas of Borcherds associated to a $0$-dimensional boundary component involve 
a choice of Weyl chamber. 

It remains to compute the terms for the other two types of orbits.

\subsection{Rank $1$ terms}   Suppose that  $\eta= [0,\eta_2]$ for $\eta_2\in \Q^2$ nonzero. 

Here we have to compute a regularization of 
$$\sum_{\substack{ \eta=[0,\eta_2]}}  \int_{\Gamma'_\infty\back\H} \big(F(g'_\tau),\theta_\eta(g'_\tau,\ph_\infty)\big)\,v^{-2}\,du\,dv.$$
This comes to taking the constant term at $s=0$ of the sum on $\eta_2$ of the integrals
$$  \sum_m  \sum_{x_0\in V_0(\Q)}\int_{\Gamma'_\infty\back \H} \hat c(m)\cdot \widehat{\ph_{\tau,z}}\big (x_0, 0,\eta_2)\,q^m\,\,v^{-s-\ell/2-2}\,du\,dv.$$

Now in the integrand
$$B\dop \eta_\tau = -(v_{01},x_0)\,a  - (v_{02},x_0) \,b$$
is independent of $\tau$, 
and we have
$$\widehat{\ph_{\tau,z}}(x_0,\eta) = v^{\frac{n-2}4}\,v_2 \, e(Q(x_0) \tau)\,
e( B\dop \eta_\tau) \, \exp(-\pi v^{-1} v_1^{-1} v_2 \, |b\tau_1-a|^2).$$
Thus, 
\begin{align*}
&\sum_m  \sum_{x_0\in V_0(\Q)}\int_{\Gamma'_\infty\back \H}\hat c(m)(x_0,0,\eta_2)\, v^{\frac{n-2}4}\,v_2 \, e(Q(x_0) \tau)\\
\nass
{}&\qquad\qquad\qquad \times e( B\dop \eta_\tau) \, \exp(-\pi v^{-1} v_1^{-1} v_2 \, |b\tau_1-a|^2)\,e(m\tau)\,\,v^{-s-\ell/2-2}\,du\,dv\\
\nass
\nass
{}&=\Gamma(s+1)\,(\pi  v_1^{-1} v_2)^{-s-1}\,v_2\sum_m  \sum_{\substack{x_0\in V_0(\Q)\\ \snass Q(x_0)=m}}\hat c(-m)(x_0,0,\eta_2)\,e( B\dop \eta_\tau) \, \, |b\tau_1-a|^{-2s-2}
.
\end{align*}

The sum here is finite, since only a finite number of $\hat c(-m)$ for $m\ge 0$ are nonzero and the corresponding set of $x_0$'s is also finite. 

We must still sum on $\eta$.   Again we take $L$ to be the lattice defined in section~\ref{subsection-lattice}. Then by (\ref{hatcm}) and (\ref{hatph}), 
$$\hat c(m)(x_0,0,\eta_2) = \sum_{\substack{\l\in L^\vee/L \\ \snass \l_1=0}} c_\l(-m)\,e(-\l_2\dop \eta_2)\,\ph_{\l_0}(x_0)\,\ph_{\hat\Z^2}(\eta_2),$$
so that $\eta_2$ will run over non-zero elements of $\Z^2$.  Note that, by taking the sum in this way, we are implicitly including the factor of $2$ 
coming from the identical contributions of $\gamma$ and $-\gamma$ in the unfolding. 
For fixed $m$, $\l$  and $x_0$, we must compute
the constant term at $s=0$ of
\be \label{basic-II}
\Gamma(s+1)(\pi  v_1^{-1} v_2)^{-s-1}\,v_2\,\sum_{a, b}^{\prime} e( C_0 a + C_1 b) \, \, |b\tau_1-a|^{-2s-2}
\ee
where 
$$C_0= -(v_{01},x_0)-\l_{21}, \qquad C_1 = - (v_{02},x_0)-\l_{22}.$$

First suppose that  $C_0$ and $C_1$ are not both zero. Note that, if $m\ne0$ so that $x_0\ne 0$, this will 
generically be the case.   We can apply the second Kronecker limit formula, Siegel \cite{siegel}, (39), p.32, 
$$\frac{z-\bar z}{-2\pi i} \sum'_{m,n} \frac{e^{2\pi i(mu+nv)}}{|m+nz|^2} = \log\bigg\vert \frac{\vartheta_1(v-uz,z)}{\eta(z)}\,e^{\pi i z u^2}\bigg\vert^2.$$

Setting $s=0$ in (\ref{basic-II}), we have
$$
\pi^{-1}  v_1 \,\sum_{a, b}^{\prime} e( C_0 a + C_1 b) \, \, |b\tau_1-a|^{-2}
= 
-\log\bigg\vert \frac{\vartheta_1(C_1+ C_0\tau_1,\tau_1)}{\eta(\tau_1)}\,e^{\pi i \tau_1 C_0^2}\bigg\vert^2.
$$
Here note that 
$$C_1+C_0\tau_1 = -(x_0,w_0) -\L_2,  \qquad{\text{where}}\quad  \L_2 = \l_2\dop \bpm \tau_1\\1\epm = \l_{21}\tau_1+\l_{22}.$$

The full contribution of these terms is then 
\be\label{rank1-generic-1}
-\sum_m \sum_{\substack{\l \\ \snass \l_1=0}}
c_\l(-m) \sum_{\substack{x_0\in \l_0+L_0\\ \snass Q(x_0)=m}} \log\bigg\vert \frac{\vartheta_1(-(x_0,w_0)-\L_2,\tau_1)}{\eta(\tau_1)}\,e^{\pi i \tau_1 C_0^2}\bigg\vert^2.
\ee

Recall that the theta series 
\be\label{jacobi-I}
\vartheta_1(z,\tau) = \sum_{n\in \Z}  e^{i\pi(n+\frac12)^2\tau + 2\pi i(n+\frac12)(z-\frac12)},
\ee
has a product expansion, \cite{siegel}, (36), p30, 
\be\label{jacobi-II}
\vartheta_1(z,\tau) = -i e^{i\pi (\tau/4)} (e^{i\pi z} - e^{- i \pi z}) \prod_{n=1}^\infty (1-e^{2\pi i (z+n\tau)})\,(1-e^{-2\pi i (z-n\tau)})(1-e^{2\pi i n \tau}).
\ee

We may then write the contribution of these rank $1$-orbits as $-\log|\ |^2$ of the following product
$$\prod_m\prod_{\substack{\l \\ \snass \l_1=0}}
\bigg(\ \prod_{\substack{x_0\in \l_0+L_0\\ \snass Q(x_0)=m}}  \frac{\vartheta_1(-(x_0,w_0)-\L_2,\tau_1)}{\eta(\tau_1)}\,e^{\pi i \tau_1 C_0^2}\ \bigg)^{c_\l(-m)},
$$
or in a fully expanded version which will be useful in section~\ref{section-compare}
\begin{multline}\label{factor-II-expanded}
\prod_m  \prod_{\substack{\l \\ \snass \l_1=0}}\bigg(\  \prod_{\substack{x_0\in \l_0+ L_0\\ \snass Q(x_0)=m}} \,e^{\pi i \tau_1 C_0^2}\,
(e(\frac12((x_0,w_0)+\L_2)) - e(-\frac12((x_0,w_0)+\L_2))\,\\
\nass
\times q_1^{\frac{1}{12}} \,
\prod_{n=1}^\infty \big(1- e(-(x_0,w_0)-\L_2)\,q_1^n\,\big)\,\big(1-e((x_0,w_0)+\L_2)\,q_1^n\,\big)  \  \bigg)^{c_\l(-m)}.
\end{multline}
Here we will want to extract the factor
\be 
\prod_m  \prod_{\substack{\l \\ \snass \l_1=0}}\bigg(\  \prod_{\substack{x_0\in \l_0+L_0\\ \snass Q(x_0)=m}} \,e^{\pi i \tau_1 C_0^2\,}\,\bigg)^{c_\l(-m)}
\ee
whose $-\log|\ |^2$ is 
\be\label{extra-II}
2\pi v_1\sum_{m}\sum_{\substack{\l \\ \snass \l_1=0}}\sum_{\substack{x_0\in \l_0+L_0\\ \snass Q(x_0)=m}} c_\l(-m)\, ((x_0, v_{01})+\l_{21})^2.
\ee

Next suppose that $C_0 = C_1=0$. This will always occur when $\l=0$ and $m=0$, so that $x_0=0$.  It can also occur when $m\ne 0$ 
and $(w_0,\tau_1)$ lies on certain affine hyperplanes. In this case (\ref{basic-II}) reduces to the Eisenstein series, and we have
\begin{align} \label{basic-II-0}
&\Gamma(s+1)\pi^{-s-1}\,v_2^{-s}\, v_1^{s+1}\,\sum_{a, b}^{\prime}  |b\tau_1-a|^{-2s-2}\\
\nass
{}&= \Gamma(s+1)\,\pi^{-s-1}\, v_2^{-s} \bigg(\ \frac{\pi}{s} + 2 \pi (\ \gamma-\log 2 - \log(v_1^{\frac12}|\eta(\tau_1)|^2) + O(s)\ \bigg),\notag
\end{align}
by the first Kronecker limit formula, \cite{siegel}, p.14.   This has a pole with residue $1$ at $s=0$ and the constant term there is 
\be\label{rank1-sing}
\gamma-\log(4\pi v_2) - 2 \log(v_1^{\frac12}|\eta(\tau_1)|^2).
\ee
Thus, in the generic case, i.e., when $(w_0,\tau_1)$ is not on any singular hyperplane, we obtain an additional contribution:
\be\label{weight.term}
-c_0(0)\,\big(\, \log(4\pi v_1v_2) -\gamma + 2 \log|\eta(\tau_1)|^2\, \big).
\ee
Note that the quantity $-c_0(0)\,(\,\log(4\pi v_1v_2) -\gamma)$ is part of the normalized Petersson inner product in (\ref{borch.form}). 

\subsection{The zero orbit}\label{sec.zero.orbit}    Finally, we have the term for $\eta=0$. 
In this case, 
$$\widehat{\ph_{\tau,z}}(x_0,0) = v^{\frac{n-2}4}\,v_2 \, e(Q(x_0) \tau).$$
and we need to compute
\be\label{type-I.contrib}
  \int_{\Gamma'\back\H}^{\text{reg}} \bgs{F(g'_\tau)}{\theta_0(g'_\tau,\ph_\infty)}\,v^{-2}\,du\,dv.
  \ee
This integral is essentially the Rankin product of $F$ with a positive definite theta series attached to $V_0$. 
More precisely, write
\be\label{full-trans-F}
F^o(\tau) = \sum_m \hat c(m)(\cdot, 0)\, q^m = \sum_m \sum_{\l} c_\l(m)\,q^m\, \hat \ph_\l(\cdot, 0),
\ee
so that $F^o: \H \lra S_{L_0}\subset S(V_0(\A_f))$ is a weakly holomorphic form of weight $-\ell = 1-\frac{n}2$. 
Also note that only terms with $\l_1=0$ contribute to this sum, and that, for such a $\l$, 
$$\hat\ph_\l(x_0,0) = \ph_{\l_0}(x_0).$$
Thus, 
$$F^o(\tau) = \sum_{\l_0} F^o_{\l_0}(\tau)\, \ph_{\l_0}$$
where
$$F^o_{\l_0}(\tau) = \sum_m \sum_{\l_2} c_{\l_0+\l_2}(m)\, q^m.$$
For $\l_0\in L_0^\vee/L_0$, we have a theta series of weight $\ell$
$$\theta(\tau, \ph_{\l_0}) = \sum_{x_0\in\l_0+ L_0} e(Q(x_0)\tau).$$ 
By Corollary~9.3 of \cite{borch95}, (\ref{type-I.contrib}) is equal to
\vskip -14pt
$$
\frac{\pi}3\,v_2\,\CT{}[ \,E_2(\tau)\,\sum_{\substack{\l_0}} F^o_{\l_0}(\tau)\,\theta(\tau, \ph_{\l_0})\,] = - 8 \pi\, v_2\,\sum_{m} \sum_{\substack{\l \\ \snass \l_1=0}}
\sum_{x_0\in \l_0+L_0} c_\l(-m)\,\s_1(m-Q(x_0)),
$$
where $\CT{}$ means the constant term in the $q$-expansion and
$$E_2(\tau) = 1- 24\sum_{m=1}^\infty \s_1(m)\, q^m.$$
We set $\s_1(0) = -\frac1{24}$ and $\s_1(r)=0$ for $r\notin \Z_{\ge0}$.   In particular, only terms with $m\ge 0$ occur and the sum is finite. 
For convenience, we write
\be\label{main-I-alt} 
I_0 := -  \sum_{m} \sum_{\substack{\l \\ \snass \l_1=0}}
\sum_{x_0\in \l_0+L_0} c_\l(-m)\,\s_1(m-Q(x_0))  = \CT{}[ \,E_2(\tau)\,\sum_{\substack{\l_0}} F^o_{\l_0}(\tau)\,\theta(\tau, \ph_{\l_0})/24\,] 
\ee
so that the contribution from $\eta=0$ is simply  $8\pi v_2\,I_0$.    Note that $24\,I_0$ is an integer, since the Fourier coefficients $c_\l(-m)$ for $m\ge 0$ of the original 
input form $F$ are required to be integers.

\subsection{Borcherds' vector system identity}\label{borch.quad}
 At this point, to obtain our final formula, we need to combine the contribution (\ref{type-I.contrib}), in the form just given,   
with the quantity (\ref{extra-II}),  using a version of Borcherds'  vector system 
identity, 
\cite{borch98}, p.536, Theorem~10.5.
In order to describe this identity in our present case, we consider another partial Fourier transform map. Let
$$V_{00} = \Q e_1 + V_0 + \Q e_1'$$
so that $V_{00}$ has signature $(n-1,1)$ and we have a Witt decomposition
$$V = \Q e_2 + V_{00} + \Q e_2'.$$
Define a map
\be\label{partialFT-II} 
S(V(\A_f)) \lra S(V_{00}(\A_f)), \qquad \ph \mapsto \hat\ph^{oo},
\ee
where
$$\hat \ph^{oo}(x_{00}) = \int_{\A_f} \ph(x_{00}+ y\,e_2)\, dy.$$

Let 
$$L_{00} = \Z\, e_1+ L_0 +  N\Z e_1'$$
so that $L_{00}$ has signature $(n-1,1)$ and 
$$L =  \Z \,e_2 +L_{00}+ N\Z \,e_2'.$$
Also, by analogy with (\ref{hatph}), for $\l= \l_0+\l_1+\l_2$, we have
$$\hat \ph_\l^{oo}(x_0+ a e_1 + a' e_1') = \ph_{\l_0}(x_0)\ph_{\l_{11}}(a')\ph_{\l_{21}}(a)\, \ph_{\l_{12}}(0).$$
Let  $\l_{00} = \l_0 + \l_{21}e_1 + \l_{11}e_1'$ and set 
$$c_{\l_{00}}(m) = \sum_{\substack{\l= \l_{00} + \l_{22} e_2}} c_\l(m).$$
Then the image of $c(m) \in S_L$ under the partial Fourier transform (\ref{partialFT-II}) is  
$$\hat c^{oo}(m) = \sum_{\l_{00}} c_{\l_{00}}(m)\, \ph_{\l_{00}},$$
and the image 
of $F$ under this partial Fourier transform is an $S_{L_{00}}$-valued weakly holomorphic modular form $F^{oo}$ 
with Fourier expansion 
$$ F^{oo}(\tau) = \sum_{m} \sum_{\l_{00}\in L_{00}^\vee/L_{00}}\,c_{\l_{00}}(m) \, q^m\,\ph_{\l_{00}}.$$
Note that the function $F^o$ of (\ref{full-trans-F}) can be obtained from $F^{oo}$ by applying a second partial Fourier transform. 
As explained in \cite{borch98}, p.536, the fact that the Borcherds lift of $F^{oo}$ defines a piecewise linear function on the negative cone in $V_{00}(\R)$
amounts to the following relation for all vectors $v_{01}\in V_0(\R)$.
\begin{prop}\label{borch.vectorsystem} (Borcherds' vector system identity)
\be \label{borch-quad-rel}
4\, I_0\cdot Q(v_{01}) =\sum_{m>0} \sum_{\substack{\l_{00}\\ \snass \l_{11}=0}} c_{\l_{00}}(-m) \sum_{\substack{x_0\in\l_0 + L_0\\ \snass Q(x_0)=m}} \,(x_0,v_{01})^2.
\ee
\end{prop}
We can rewrite this in terms of the original coefficients as follows.
\begin{cor} \label{borch-quad-rel-II}  For any $v_{01}\in V_0(\R)$, 
$$4\, I_0\cdot Q(v_{01}) =\sum_{m>0} \sum_{\substack{\l\\ \snass \l_{1}=0}} c_{\l}(-m) \sum_{\substack{x_0\in\l_0 + L_0\\ \snass Q(x_0)=m}} \,(x_0,v_{01})^2.$$
\end{cor}

Thus the sum of (\ref{extra-II}) and (\ref{type-I.contrib}) is 
\be\label{vanishing}
8\pi (v_1 Q(v_{01})  + v_2 )\,I_0= 8\pi v_2' \,I_0
\ee
plus the additional term 
\be\label{extra-junk}
2\pi \sum_{m}\sum_{\substack{\l \\ \snass \l_1=0}}\sum_{\substack{x_0\in \l_0+L_0\\ \snass Q(x_0)=m}} c_\l(-m)\, \l_{21}\big(\ 2(x_0, v_{01})v_1+\l_{21}v_1\ \big).
\ee
Thus, this quantity is $-2\log|\  |^2$ of 
$$q_2^{I_0} \,\prod_m\prod_{\substack{\l \\ \snass \l_1=0}}\bigg(\ \prod_{\substack{x_0\in \l_0+L_0\\ \snass Q(x_0)=m}} e(\,(x_0,w_0) + \frac12\,\L_2\,)\ \bigg)^{c_\l(-m)\l_{21}/2}.$$

Collecting all contributions, we obtain the result stated in Theorem~\ref{mainthm}.

\section{Examples}

{\bf 1.} In the simplest case where $L$ is self-dual, we consider a weakly holomorphic form  
$F= \sum_m c_0(m)\, q^m\,\ph_0,$
with corresponding Borcherds form $\Psi(F)$.  Suppose that $L$ has a
Witt decomposition as in (\ref{def-L}), with $N=1$.  Then, our product formula for the Borcherds form $\Psi(F)$
reduces to that given in Theorem A of the introduction. 
Let 
\be\label{V0-roots} 
R_0(F)= \{ \a_0\in L_0\mid Q(\a_0)>0, \ c_0(-Q(\a_0))\ne 0\ \},
\ee
and let $W_0$ be a connected component of the complement of the hyperplanes, $\a_0^\perp$, $\a_0\in R_0(F)$, in $V_0(\R)$. 
Let  we can write
the factor (\ref{example1-second-II}) as 
\be\label{}
\pm\, i^{B}\,q_2^{I_0}\,\eta(\tau_1)^{c_o(0)}\,
\prod_{\substack{x_0\in L_0\\  \snass (x_0,W_0)>0}}  \bigg(\ \frac{\vartheta_1(-(x_0,w_0),\tau_1)}{\eta(\tau_1)}
\ \bigg)^{c_o(-Q(x_0))},
\ee
where 
\be\label{def-B}
B=\frac12 \sum_{\substack{x_0\in L_0\\ \snass x_0\ne 0}} c_o(-Q(x_0)).
\ee

{\bf 1.0.} The simplest case of all is when $L_0=0$ and $F_o(\tau) = j(\tau)-744 = q^{-1} + O(q)$. 
In this case, $I_0=-1$ and $c_o(0)=0$, so that our product reduces to 
$$j(\tau_2)-j(\tau_1) = q_2^{-1} \, \prod_{a>0}\prod_b (1-q_2^a \,q_1^b)^{c_o(ab)}.$$
This is the example mentioned on p.163 of \cite{borch95}.  The left side of this identity is a meromorphic 
function on $X(1)\times X(1)$, where $X(1) = \SL_2(\Z)\back \H^*$ is the compactified modular curve.
Once the factor $q_2^{-1}$ has been removed, the remaining product is convergent  for $v_1$ in a bounded set and $v_2$ large, i.e., 
in a neighborhood of the $1$-dimensional boundary component $Y(1) \times \{i\infty\}$. 

{\bf 1.1.}  One of the most beautiful examples is the case $L = \Pi_{26,2}$ and 
$$F_o(\tau) = \eta(\tau)^{-24} = q^{-1} \sum_{r=0}^\infty p_{24}(r)\,q^r = q^{-1} + 24 + 324\, q + \dots,$$
so that $\Psi(F)$ has weight $12$. Here we recover some of the results of 
\cite{gritsenko.borch24}.  For any positive definite even unimodular lattice $L_0$ of rank $24$, i.e., any Niemeier lattice, 
we have an isomorphism
$L \simeq L_0 + H^2,$
where $H$ is a rank $2$ hyperbolic lattice. It follows that, up to the action of $\text{\rm Aut}(L)$,  there are $24$ such decompositions, determined by the isometry class of 
$L_0$,  and we obtain 
a product formula for $\Psi(F)$ for each of them.  Analogously, there is only one orbit of Witt decompositions of the form $L= L_{00}+H$ and associated 
$0$-dimensional boundary component. 
Let $N_2(L_0)=24\,h$ be the number of lattice vectors of norm $2$ in $L_0$; the values of $h$ are listed Table~16.1, p.407 
of \cite{conway.sloane}.  The quantity $I_0$ is given by 
$$I_0 = \frac{1}{24} N_2(L_0) = h.$$ 
Up to a scalar of absolute value $1$, in a neighborhood of the $1$-dimensional cusp associated to $L_0$,  $\Psi(F)$ is the product of two factors, 
\be\label{smear-24}
\prod_{a=1}^\infty\prod_{b\in \Z}
\prod_{\substack{x_0\in L_0}} 
\big(1- q_2^a\, q_1^b\, e( - (x_0, w_0))\,\big)^{c_o(ab-Q(x_0))},
\ee
and 
\be\label{second-24} 
q_2^{ h}\, \eta(\tau_1)^{24- h}\,\prod_{\substack{x_0\in L_0\\ \snass Q(x_0)=1\\ \snass (x_0,W_0)>0}}\vartheta_1(-(x_0,w_0),\tau_1),
\ee
where $W_0$ is any connected component
of the complement of the hyperplanes $x_0^\perp$ in $V_0(\R)$ as $x_0$ runs over the vectors with $Q(x_0)=1$ in $L_0$, the `roots' of $L_0$. 
For example, $h=0$ only when $L_0$ is the Leech lattice and, in this case, the second factor reduces to $\eta(\tau_1)^{24}$.

It would be most natural to normalize $\Psi(F)$ by taking it to be equal to
$\eta(\tau_1)^{24}$ times the second factor (\ref{smear-24}) in the neighborhood of the boundary component corresponding to the Leech lattice. 
It is then an interesting question to determine the scalar factors arising in the other product expansions.  

The compactifying divisor for the $1$-dimensional boundary component indexed by a lattice $L_0$ is\footnote{Up to orbifold aspects.} 
the abelian scheme
$$\Cal E(L_0):=L_0\tt_\Z \Cal E \lra Y(1)= \SL_2(\Z)\back \H,$$
of relative dimension $24$, 
where $\Cal E \lra Y(1)$ is the universal elliptic curve and the tensor product is the Serre construction. The function (\ref{second-24}), 
with the $q_2$ factor removed, is a section of a certain line bundle over $\Cal E(L_0)$.  For example, in the case of the Leech lattice, this bundle is just the pullback from the base of the 
line bundle of modular forms of weight $12$. 
In general, the Borcherds form $\Psi(F)$ extends to the smooth toroidal (partial) compactification obtained by adding these compactifying divisors 
for the $1$-dimensional boundary components and the multiplicity of the divisor associated to a lattice $L_0$ in $\div(\Psi(F))$ is $h = N_2(L_0)/24$, the Coxeter number of $L_0$. 
The theta function occurring in (\ref{second-24}) is the analogue of that considered by Looijenga, \cite{looijenga-root}, p.31,  in the case of a root lattice. Its divisor
is the union of the `root' hypertori and is invariant under the group $\text{\rm Aut}(L_0)$,  whose natural action on $\Cal E(L_0)$ extends to the relevant line bundle. 

{\bf 2.} We consider the example of Gritsenko-Nikulin, \cite{grit.nik},  discussed in section 5 of \cite{Bints}.   In this case, 
we have $L = \Z^5$ with inner product defined by (\ref{sig-32}), so that the signature is $(3,2)$, 
$N=1$ and $L_0 = \langle 2\rangle$.  The $S_L$-valued input form $F$ is obtained from the Jacobi form $\phi_{0,1}(\tau,z) = \phi_{12,1}(\tau,z)\eta(\tau)^{-24}$, 
cf. (5.27) of \cite{Bints}. 
It 
has weight  $-\frac12$ and the associated Borcherds form $\Psi(F)$ is $2^{-6}\,\Delta_5(z)$, where $\Delta_5$ is the Siegel cusp form of weight $5$.   
Here $L^\vee/L = L_0^\vee/L^{\phantom{\vee}}_0$ so that $\l_1=0$, $\l_2=0$, and $\l_0=0$ or $\frac12$. We write $\ph_0$ and $\ph_1$ for the 
corresponding coset functions and let
$$F = F_0 \ph_0 + F_1\ph_1, \qquad \quad
F_0(\tau) = 10+108 q+ \dots, \qquad F_1(\tau) = q^{-\frac14} - 64 q^{\frac34} +\dots.$$
We then find that $I_0=1/2$ and that the product formula of Theorem~\ref{mainthm} for $\Psi(F)$ reduces to 
\be\label{first-GN}
i \,\eta(\tau_1)^{10}\,q_2^{\frac12}\, \frac{\vartheta(w_0,\tau_1)}{\eta(\tau_1)}\,\prod_{\substack{(a,r,b)\in \Z^3\\ \snass a>0}}
\big(1- q_0^r \,q_1^b\,q_2^a\,\big)^{c(r^2-4ab)}.
\ee
Here $q_0= e(w_0)$, $r = 2 x_0$,  and we use the convention that, for an integer $d$ congruent to $0$ or $1$ 
modulo $4$,  $c(d) = c_0(-d/4)$ for $d\equiv 0\mod 4$ and $c(d) = c_1(-d/4)$ for $d \equiv 1\mod 4$. 
This is essentially equivalent to the product formula given in \cite{GK-II}, (2.7), p.234, and (2.16), p.239, noting that 
the Fourier coefficients of their Jacobi form are given by the relation $f(n,\ell) = c(\ell^2-4n)$. 

\section{Comparison}\label{section-compare}

In this section, we explain the relation between our product formula, associated to an isotropic $2$-plane and that of Borcherds,  
associated to an isotropic line and a particular choice of Weyl chamber.  

Suppose that $\ell$ is an isotropic line in $V$ which is contained in an isotropic plane $U$.  If $M$ is an even integral lattice in $V$, 
we take basis $e_1$ and $e_2$ for $M_U= M\cap U$ such that $\ell\cap M = \Z e_2$. 
We then get compatible Witt decompositions (\ref{eq1}) and 
\be\label{Witt-1}
V= \ell + V_{00} + \ell'
\ee
where $\ell'=\Q e_2'$ and 
\be\label{Witt-2}
V_{00} = \Q e_1+ V_0 + \Q e_1'.
\ee
As explained in section~\ref{subsection-lattice}, we may choose a lattice $L\subset M$ with $L_U=M_U$ compatible with these Witt decompositions. 
With respect to (\ref{Witt-1}) and (\ref{Witt-2}), a vector $x$ with coordinates as in (\ref{x-coord}) becomes
$$x = \bpm x_{22} \\ x_{00} \\ x_{12}\epm, \qquad x_{00} = \bpm x_{21}\\ x_0\\ x_{11}\epm \in V_{00}(\Q).$$
Now our vector $w$ as in (\ref{w-coords})
can be written as
$$w = \frak z +  e'_2 - Q(\frak z)\, e_2, \qquad \frak z\in V_{00}(\C),$$
so that 
$$\frak z = \bpm -\tau_2'\\ w_0 \\ \tau_1\epm, \qquad Q(\frak z) =Q(w_0) -\tau_1\tau_2'.$$

For simplicity, we assume that $L=M$ is unimodular so that 
$$L = \Z e_2 + L_{00} + \Z e_2', \quad\text{with}\quad L_{00} = \Z e_1 + L_0 + \Z e_1'.$$

Note that 
$$\frac{\vartheta_1(z, \tau)}{\eta(\tau)} = i\, q^{\frac1{12}} \,e(-\frac12z)\,(1-e(z))\,\prod_{n=1}^\infty (1-q^n \,e(z))(1-q^n\, e(-z)).$$
Then we can write the product (\ref{example1-second-II}) as the product of the quantities
\be\label{B-factor-2}
\prod_{b>0}
\prod_{\substack{x_0\in L_0}}  (1- q_1^b \, e(-(x_0,w_0))\ \big)^{c_o(-Q(x_0))},
\ee
\be\label{B-factor-3}
\prod_{\substack{x_0\in L_0\\ \snass (x_0, W_0)> 0}} (1-e(-(x_0,w_0))\ \big)^{c_o(-Q(x_0))},
\ee
and 
\be\label{B-factor-4}
(-1)^{B/2}i^{B}\,q_1^{\frac{1}{24} \, c_o(0)+\frac1{12}B}\, q_2^{I_0} \, \prod_{\substack{x_0\in L_0\\ \snass (x_0, W_0)> 0}} e((x_0,w_0))^{c_o(-Q(x_0))/2},
\ee
where $B$ is given by (\ref{def-B}). Note that, in each case, the product on $x_0$ is taken over a finite set of vectors.

To relate this product expansion to that of Borcherds, we need some information about his Weyl chambers. 
Let 
\be\label{V00-roots} 
R_{00}(F) = \{ \a\in L_{00} \mid Q(\a)>0, \ c(-Q(\a))\ne 0\ \},
\ee
be the set of `roots' in $L_{00}$ for $F$.  
The walls  in $V_{00}(\R)$ are the hyperplanes $\a^\perp$ given by $(\a,y)=0$ for $\a\in R_{00}(F)$. 
Let $C_{00}$ be the component of cone of negative vectors in $V_{00}(\R)$ determined by $D$.
The Weyl chambers in  Borcherds are the connected components of the complement
$$C_{00} - \bigcup_{\a\in R_{00}(F)} C_{00}\cap \a^\perp.$$ 

Let $m_{\max}$ be the largest positive integer such that $c_o(-m)\ne 0$,
and let $W_0$ be a connected component of the set 
$$V_0(\R) - \bigcup_{\a_0\in R_0(F)} \a_0^\perp,$$
where $R_0(F)$ is given by (\ref{V0-roots}). 
Let 
$$R_0(F)^+ = \{ \a\in R_0(F)\mid (\a_0,W_0)>0\ \},$$
so that 
$$R_0(F) = R_0(F)^+ \sqcup (-R_0(F)^+).$$

The crucial facts for us are the following.
\begin{lem}  There is a unique Weyl chamber $W_{00}$ in $C_{00}$ containing a vector $y$ with $y_1=1$, $y_2> 4 m_{\max}+2$, and with 
$$0<(\a_0,y_0)<\frac12, \quad \forall \a_0\in R_0(F)^-.$$
\end{lem}

\begin{lem} For the Weyl chamber $W_{00}$ characterized in the previous lemma,  the set
$$\left\{\ x_{00} = \bpm b \\ -x_0 \\ -a\epm \in L_{00}\ \bigg\vert\  c_o(-Q(x_{00}))\ne 0, \qquad (x_{00},W_{00})>0\right \}$$
is given by 
$$\left\{\ x_{00} = \bpm b \\ -x_0 \\ -a\epm \in L_{00}\ \bigg\vert\  c_o(-Q(x_{00}))\ne 0, \text{and}\quad \begin{matrix}\text{$a>0$, or $a=0$, $b>0$, }\\ 
\text{or}\\ \text{$a=b=0$ and $x_0\in R_0(F)^+$}
\end{matrix}\right \}.$$
\end{lem}

Noting that $Q(x_{00}) = Q(x_0)-ab$ and that
$$(x_{00}, \frak z) = -(x_0,w_0) +a \tau_2' + b \tau_1,$$
we can write the product of the factors (\ref{example1-first}), (\ref{B-factor-2}), (\ref{B-factor-3}), and (\ref{B-factor-4})  
as
\be\label{Borcherds-prod}
(-1)^{B/2}i^{B}\,e((\rho_{00}, \frak z))\,\prod_{\substack{x_{00}\in L_{00}\\ \snass (x_{00},W_{00})>0}} \big(\ 1-e((x_{00},\frak z))\ \big)^{c_o(-Q(x_{00}))},
\ee
where $\rho_{00}$ is the `Weyl vector'
$$\rho_{00} =  \frac12 \sum_{\substack{x_0\in L_0 \\ \snass (x_0,W_0)>0}} c_o(-Q(x_0)) \, x_0 - \frac12 I_0\,e_1' + \frac1{24}(c_o(0)+ 2B) e_1.$$
associated to $W_{00}$. 
This is precisely
the product of Theorem~13.3 in Borcherds \cite{borch98} 
with respect to the Weyl chamber $W_{00}$ or Theorem~10.1 of \cite{borch95}.  Note that, up to some differences in sign conventions, 
our Weyl vector coincides with that of Theorem~10.4 of \cite{borch95}.  In particular, the vector system in $V_{00}$ associated to $F^{oo}$
of section~\ref{borch.quad}, has index 
$\text{\bf m} = I_0$, via (\ref{main-I-alt}),  and `dimension'  $\text{\bf d}= c_o(0) + 2 B$, where these invariants are explained in section 6 of \cite{borch95}.

\end{document}